\documentclass[a4paper, 11pt, reqno, english]{smfart}

\usepackage{aeguill}
\usepackage{enumerate}
\usepackage{amssymb,amsmath,latexsym,amsthm}
\usepackage[T1]{fontenc}
\usepackage[a4paper, tmargin=1.5in, bmargin=1.5in, lmargin=1in,rmargin=1in]{geometry}
\usepackage{url}
\usepackage[french, main=english]{babel}
\usepackage[utf8]{inputenc}
\usepackage{mathrsfs}
\usepackage{xcolor}
\usepackage{comment}
\definecolor{violet}{rgb}{0.0,0.2,0.7}
\definecolor{rouge2}{rgb}{0.8,0.0,0.2}
\usepackage{tikz}
\usepackage{empheq}
\usepackage{tikz-cd}
\usetikzlibrary{matrix,arrows,decorations.pathmorphing}
\usepackage{hyperref}
\usepackage{mathpazo}
\hypersetup{
	bookmarks=true,         
	unicode=false,          
	pdftoolbar=true,        
	pdfmenubar=true,        
	pdffitwindow=false,     
	pdfstartview={FitH},    
	pdftitle={},    
	pdfauthor={},     
	colorlinks=true,       
	linkcolor=rouge2,          
	citecolor=violet,        
	filecolor=black,      
	urlcolor=cyan}           
\usepackage{enumitem}
\usepackage{appendix}

\usepackage{mathtools}
\usepackage{dsfont}

\theoremstyle{plain}
\newtheorem{thm}{Theorem}[section]

\newtheorem{lem}[thm]{Lemma}
\newtheorem{prop}[thm]{Proposition}
\newtheorem{defnprop}[thm]{Definition and Proposition}

\theoremstyle{plain}
\newtheorem{bigthm}{Theorem}
\newtheorem{bigprop}[bigthm]{Proposition}

\newtheorem{taggedbigsetx}{Setting} 
\newenvironment{taggedbigset}[1]
{\renewcommand\thetaggedbigsetx{#1}\taggedbigsetx}
{\endtaggedbigsetx}

\newtheorem{taggedvolcapx}{Volume-capacity comparison}
\newenvironment{taggedvolcap}[1]
{\renewcommand\thetaggedvolcapx{#1}\taggedvolcapx}
{\endtaggedvolcapx}
\newtheorem{taggedgeoconstx}{Geometric constants}
\newenvironment{taggedgeoconst}[1]
{\renewcommand\thetaggedgeoconstx{#1}\taggedgeoconstx}
{\endtaggedgeoconstx}
\newtheorem{taggedanaconstx}{Analytic constants}
\newenvironment{taggedanaconst}[1]
{\renewcommand\thetaggedanaconstx{#1}\taggedanaconstx}
{\endtaggedanaconstx}
\newtheorem{taggedbiggsx}{Geometric setting}
\newenvironment{taggedbiggs}[1]
{\renewcommand\thetaggedbiggsx{#1}\taggedbiggsx}
{\endtaggedbiggsx}

\newtheorem{taggedbiggax}{Geometric assumption}
\newenvironment{taggedbigga}[1]
{\renewcommand\thetaggedbiggax{#1}\taggedbiggax}
{\endtaggedbiggax}

\newtheorem{taggedbigconjx}{Conjecture}
\newenvironment{taggedbigconj}[1]
{\renewcommand\thetaggedbigconjx{#1}\taggedbigconjx}
{\endtaggedbigconjx}

\newtheorem*{bigrmk*}{Remark}

\theoremstyle{definition}
\newtheorem{defn}[thm]{Definition}

\newtheorem{nota}[thm]{Notation}

\newtheorem{asm}[thm]{Assumption}

\newtheorem*{claim*}{Claim}
\newtheorem*{ack*}{Acknowledgements}
\theoremstyle{remark}
\newtheorem{rmk}[thm]{Remark}
\newtheorem*{rmk*}{Remark}

\newtheorem*{ques*}{Question}
\newtheorem*{ans*}{Answer}

\numberwithin{equation}{section}
\renewcommand{\labelenumi}{(\roman{enumi})}
\newlist{steps}{enumerate}{1}
\setlist[steps, 1]{label = Step \arabic*:}
\DeclareFontFamily{U}{MnSymbolC}{}
\DeclareSymbolFont{MnSyC}{U}{MnSymbolC}{m}{n}
\DeclareFontShape{U}{MnSymbolC}{m}{n}{
	<-6>  MnSymbolC5
	<6-7>  MnSymbolC6
	<7-8>  MnSymbolC7
	<8-9>  MnSymbolC8
	<9-10> MnSymbolC9
	<10-12> MnSymbolC10
	<12->   MnSymbolC12}{}
\DeclareMathSymbol{\intprod}{\mathbin}{MnSyC}{'270}

\DeclareMathOperator{\Vol}{Vol}

\DeclareMathOperator{\supp}{supp}

\DeclareMathOperator{\dist}{dist}

\DeclareMathOperator{\PSH}{PSH}
\DeclareMathOperator{\MA}{MA}
\DeclareMathOperator{\CAP}{Cap}
\DeclareMathOperator{\CAPBT}{Cap_{BT}}

\DeclareMathOperator{\osc}{osc}

\DeclareMathOperator{\Exc}{Exc}

\def\1{\mathds{1}}
\def\e{\mathrm{e}}

\newcommand\Kol{{\mathrm{Kol}}}
\newcommand\Lap{{\mathrm{Lap}}}

\newcommand\oh{\frac{1}{2}}

\newcommand{\ii}{\mathrm{i}}

\newcommand{\loc}{\mathrm{loc}}

\newcommand{\wX}{{\widetilde{X}}}

\newcommand{\wCU}{{\widetilde{\mathcal{U}}}}
\newcommand{\wOm}{{\widetilde{\Omega}}}

\newcommand{\wvph}{{\widetilde{\varphi}}}

\newcommand{\wrho}{{\widetilde{\rho}}}

\newcommand{\wu}{{\widetilde{u}}}

\newcommand{\bBD}{\overline{\mathbb{D}}}

\newcommand\dt{\delta}

\newcommand\vep{\varepsilon}
\newcommand\vph{\varphi}

\newcommand\om{\omega}
\newcommand\ta{\theta}
\newcommand\gm{\gamma}

\newcommand\af{\alpha}
\newcommand\bt{\beta}
\newcommand\ld{\lambda}
\newcommand\zt{\zeta}

\newcommand\Om{\Omega}

\newcommand\BN{\mathbb{N}}

\newcommand\BQ{\mathbb{Q}}
\newcommand\BR{\mathbb{R}}
\newcommand\BC{\mathbb{C}}
\newcommand\BB{\mathbb{B}}

\newcommand\BD{\mathbb{D}}

\newcommand\CC{\mathcal{C}}
\newcommand\CD{\mathcal{D}}
\newcommand\CE{\mathcal{E}}
\newcommand\CF{\mathcal{F}}

\newcommand\CO{\mathcal{O}}

\newcommand\CR{\mathcal{R}}

\newcommand\CU{\mathcal{U}}
\newcommand\CV{\mathcal{V}}
\newcommand\CW{\mathcal{W}}
\newcommand\CX{\mathcal{X}}

\newcommand\CZ{\mathcal{Z}}

\newcommand\lt{\left}
\newcommand\rt{\right}

\newcommand\scr[1]{\mathscr{#1}}

\newcommand\ra{\rightarrow}

\newcommand\clb{\color{blue}}

\newcommand\pl{\partial}
\newcommand\db{\bar{\partial}}

\newcommand\ddb{\partial \bar{\partial}}

\newcommand\dd{\mathrm{d}}
\newcommand\dc{\mathrm{d}^\mathrm{c}}
\newcommand\ddc{\mathrm{d}\mathrm{d}^\mathrm{c}}

\newcommand\norm[1]{\left\lVert {#1} \right\rVert}

\newcommand\abs[1]{\left\lvert {#1} \right\rvert}

\newcommand\w{\wedge}

\newcommand\reg{\mathrm{reg}}
\newcommand\sing{\mathrm{sing}}
\newcommand\ov[1]{\overline{#1}}

\newcommand\set[2]{\left\{ {#1} \, \middle| \, {#2} \right\}}

\newcommand\res[2]{\left. {#1} \right|_{#2}} 

\newcommand\tde[1]{\tilde{#1}}

\newcommand{\RN}[1]{\textup{\uppercase\expandafter{\romannumeral#1}}}

\title{Families of singular Chern--Ricci flat metrics}
\author{Chung-Ming Pan}
\address{Institut de Math{\'e}matiques de Toulouse; UMR 5219, Universit{\'e} de Toulouse; CNRS, UPS, 118 route de Narbonne, F-31062 Toulouse Cedex 9, France \qquad\qquad\qquad\qquad\qquad\qquad\qquad}
\email{Chung\_Ming.Pan@math.univ-toulouse.fr}
\date{\today}

\begin{document}
\maketitle
\begin{abstract}
We prove uniform a priori estimates for degenerate complex Monge--Amp\`ere equations on a family of hermitian varieties.
This generalizes a theorem of Di~Nezza--Guedj--Guenancia to hermitian contexts.
The main result can be applied to study the uniform boundedness of Chern--Ricci flat potentials in conifold transitions.
\end{abstract}

\tableofcontents

\section*{Introduction}
Let $\pi: \CX \ra \BD$ be a family of hermitian varieties (i.e. irreducible, reduced, complex analytic spaces).
Recently, Di Nezza--Guedj--Guenancia \cite{DGG2020} developed the first steps of a pluripotential theory in families of K\"ahler spaces and proved uniform bounds for K\"ahler--Einstein potentials in several cases.
The main purpose of this article is to generalize their theory and establish uniform estimates for complex Monge--Amp\`ere equations in families of hermitian varieties. 

\smallskip

The complex Monge--Amp\`ere equation is a powerful tool in complex geometry.
Many interesting geometric problems (e.g. the K\"ahler--Einstein equation) can be reduced to such type of equations.
Yau's celebrated resolution of the Calabi conjecture \cite{Yau_1978} and the resolution of Yau--Tian--Donaldson conjecture on the Fano manifolds by Chen--Donaldson--Sun~\cite{Chen_Donaldson_Sun_2015} are landmarks in smooth K\"ahler--Einstein problems.

\smallskip

In recent decades, following the works of Yau \cite{Yau_1978} and Tsuji \cite{Tsuji_1988}, degenerate complex Monge--Amp\`ere equations have been intensively studied. 
The breakthrough results of Ko{\l}odziej~\cite{Kolodziej_1998} and Eyssidieux--Guedj--Zeriahi~\cite{EGZ_2009} led to many important advances.
In \cite{EGZ_2009}, Yau's theorem has been generalized to compact K\"ahler varieties with log terminal singularities.
For varieties with ample canonical divisor and semi-log canonical singularities, Berman--Guenancia~\cite{Berman_Guenancia_2014} applied the variational approach developed in \cite{BBGZ_2013} to extend Aubin--Yau's result \cite{Aubin_1978, Yau_1978} on stable varieties.
On singular Fano varieties, Berman--Boucksom--Jonsson~\cite{BBJ_2021}, Li--Tian--Wang~\cite{Li_Tian_Wang_2021}, and Li~\cite{Li_2022} built a connection between singular K\"ahler--Einstein metrics and $K$-stability.

\smallskip

In hermitian contexts, the construction of hermitian Calabi--Yau metrics (i.e. Chern--Ricci flat metrics) is more difficult because the metrics are no longer closed. 
A Chern--Ricci flat hermitian metric on a complex manifold $X$ can be constructed by solving the complex Monge--Amp\`ere equation
\begin{equation*}
	(\om + \ddc \vph)^n =  c f \dd V_X, 
	\quad\text{and}\quad 
	\sup_X \vph = 0
\end{equation*}
where
\begin{itemize}
	\item $\om$ is a smooth $(1,1)$-form,
	\item $\dd V_X$ is a smooth volume form on $X$,
	\item $f \in L^p(X, \dd V_X)$ with $p>1$,
\end{itemize}
and the pair $(\vph, c) \in \lt(\PSH(X,\om) \cap L^\infty(X)\rt) \times \BR_{>0}$ is the unknown.
When $\om$ is a hermitian metric and $f$ is a smooth positive function, Tosatti--Weinkove~\cite{Tosatti_Weinkove_2010} first showed the existence and uniqueness of the pair $(\vph, c)$ with a smooth $\vph$ to the above equation. 
For $L^p$ densities $f$, Dinew--Ko{\l}odziej~\cite{Dinew_Kolodziej_2012} used pluripotential techniques to obtain uniform $L^\infty$-estimates.
The solvability was further established by Ko{\l}odziej--Nguyen~\cite{Kolodziej_Nguyen_2015} via a stability estimate.
Recently, Guedj--Lu~\cite{Guedj_Lu_3_2021} established uniform estimates and proved the existence of solution when the $(1,1)$-form $\om$ is merely big. 
As a consequence, they generalized Tosatti--Weinkove's theorem to hermitian $\BQ$-Calabi--Yau varieties.

\smallskip

It is important to study non-K\"ahler objects and how special hermitian metrics evolve when complex structures vary.
For example, to understand moduli spaces of Calabi--Yau manifolds,
a large class of non-K\"ahler Calabi--Yau threefolds was built via conifold transitions introduced by Clemens and Friedman~\cite{Clemens_1983, Friedman_1986}. 
Reid \cite{Reid_1987} speculated that all Calabi--Yau threefolds should form a connected family by conifold transitions.
Since then, these models attracted a lot of attention (cf. \cite{Friedman_1991, Tian_1992, Rossi_2006, Rong_Zhang_2011, Ruan_Zhang_2011, Chuan_2012, Fu_Li_Yau_2012, Collins_Picard_Yau_2021} and references therein).
This is our motivation to study families of singular Chern--Ricci flat metrics. 

\smallskip

\subsection*{Uniform $L^\infty$-estimate}
Before stating our results, we first fix some geometric setting for families $\pi: \CX \ra \BD$.
\begin{taggedbiggs}{(GS)}\label{geoset:irr_family}
	Let $\CX$ be an $(n+1)$-dimensional variety.
	Suppose that $\pi: \CX \ra \BD$ is a proper surjective holomorphic map with connected fibres $X_t := \pi^{-1}(t)$ which are $n$-dimensional varieties. 
	Let $\om$ be a hermitian metric on $\CX$ in the sense of Definition~\ref{defn:metrics_and_forms}.
	For every $t \in \BD$, we define a hermitian metric $\om_t$ on the fibre $X_t$ by restriction (i.e. $\om_t = \res{\om}{X_t}$).
\end{taggedbiggs}

\smallskip

In the sequel, we always assume that families of hermitian varieties $\pi: (\CX,\om) \ra \BD$ satisfy the geometric setting \ref{geoset:irr_family}.
We also impose a sup-$L^1$ comparison (see Conjecture \ref{conj:SL}).
Under such conditions, we establish a uniform bound for solutions to complex Monge--Amp\`ere equations in families of hermitian varieties: 

\begin{bigthm}\label{bigthm:uniform_a_priori_estimate}
	Let $\pi: (\CX,\om) \ra \BD$ be a family of compact, locally irreducible, hermitian varieties and $0\leq f_t \in L^p(X_t,\om_t^n)$ be a family of densities.
	Assume that $\pi: (\CX,\om) \ra \BD$ fits into Conjecture \ref{conj:SL} and $(f_t)_{t \in \BD}$ satisfies the following integral bounds: there exist constants $c_f, C_f > 0$ such that for all $t \in \BD$,
	\begin{equation}\tag{IB}\label{integral_bound:IB}
		c_f \leq \int_{X_t} f_t^{\frac{1}{n}} \om_t^n
		\quad \text{and} \quad
		\norm{f_t}_{L^p(X_t, \om_t^n)} \leq C_f.
	\end{equation}
	For each $t \in \BD$, let the pair $(\vph_t,c_t) \in \lt(\PSH(X_t, \om_t) \cap L^\infty(X_t)\rt) \times \BR_{>0}$ be a solution to the complex Monge--Amp\`ere equation:
	\begin{equation*}
		(\om_t + \ddc_t \vph_t)^n = c_t f_t \om_t^n,
		\quad\text{and}\quad
		\sup_{X_t} \vph_t = 0.
	\end{equation*}
	Then there exists a constant $C_{\MA} = C_{\MA}(c_f,C_f,C_{SL}, \CX, \om)$ such that for all $t \in \BD_{1/2}$,
	\[
	c_t + c_t^{-1} + \norm{\vph_t}_{L^\infty} \leq C_{\MA}.
	\]
\end{bigthm}

\subsection*{Sup-$L^1$ comparison conjecture}
In pluripotential theory, there is a conjecture proposed by Di~Nezza--Guedj--Guenancia \cite[Conjecture 3.1]{DGG2020} which says that if $X_0$ is irreducible then one has the following sup-$L^1$ comparison: 

\begin{taggedbigconj}{(SL)}\label{conj:SL}
	There exists a constant $C_{SL}>0$ such that: the inequality 
	\[
	\forall \vph_t \in \PSH(X_t, \om_t),\quad
	\sup_{X_t} \vph_t - C_{SL}
	\leq \frac{1}{V_t} \int_{X_t} \vph_t \om_t^n
	\leq \sup_{X_t} \vph_t
	\]
	holds for all $t \in \BD_{1/2}$, where $V_t$ is the volume of $X_t$ with respect to the hermitian metric $\om_t$. 
\end{taggedbigconj}

In the K\"ahler setting, Di~Nezza--Guedj--Guenancia \cite[Proposition 3.3]{DGG2020} established Conjecture~\ref{conj:SL} in the following cases:
\begin{enumerate}
	\item The map $\pi$ is locally trivial or projective;
	\item The fibres $X_t$ are smooth for $t \neq 0$;
	\item The fibres $X_t$ have only isolated singularities for every $t \in \BD$.
\end{enumerate}

One should notice that the irreducibility of all the fibres is a necessary condition for the left hand side inequality in Conjecture \ref{conj:SL} (cf. \cite[Example 3.5]{DGG2020}) and it is the reason why we always assume that the fibres are irreducible in the geometric setting \ref{geoset:irr_family}.

\smallskip

To establish Conjecture \ref{conj:SL} in hermitian setting, we impose the following assumptions:
\begin{taggedbigga}{(GA)}\label{geoasm:L1_assumption}
	Suppose that $\pi: \CX \ra \BD$ is a family of hermitian varieties which satisfies the geometric setting \ref{geoset:irr_family} and one of the following conditions:
	\begin{enumerate}
		\item $\pi$ is locally trivial;
		\item $\pi: \CX \ra \BD$ is a smoothing of $X_0$ and $X_0$ has only isolated singularities. 
	\end{enumerate}
\end{taggedbigga}

Note that both conditions are naturally exclusive unless $X_0$ is smooth.
Also, if $\CX$ is smooth and $\pi$ is a submersion, then (i) holds.
Thus, the geometric assumption \ref{geoasm:L1_assumption} includes families of smooth hermitian manifolds.
Then we prove that, under the geometric assumption~\ref{geoasm:L1_assumption}, Conjecture~\ref{conj:SL} is fulfilled:
\begin{bigprop}\label{bigprop:uniform_L1_isolated_singularities}
	If $\pi: (\CX,\om) \ra \BD$ is a family of hermitian varieties satisfying the geometric assumption \ref{geoasm:L1_assumption}, then there exists a uniform constant $C_{SL}$ such that Conjecture~\ref{conj:SL} holds. 
\end{bigprop}

\subsection*{Families of Calabi--Yau varieties}
A Calabi--Yau variety $X$ is a normal variety with canonical singularities and trivial canonical bundle $K_X$. 
Reid \cite{Reid_1987} has conjectured that all Calabi--Yau threefolds should form a connected family, provided one allows conifold transitions.
Roughly speaking, the construction of a conifold transition goes as follows: contracting a collection of disjoint $(-1,-1)$-curves from a K\"ahler Calabi--Yau threefold $X$ to get a singular Calabi--Yau variety $X_0$ and then smoothing singularities of $X_0$, one obtains a family of Calabi--Yau threefolds $(X_t)_{t \neq 0}$ which are non-K\"ahler for a general $t$.

\smallskip

In the model of conifold transitions, 
the central fibre $X_0$ has only ordinary double point singularities which are canonical. 
Based on these geometric models, it is thus legitimate to study a smoothing family of Calabi--Yau varieties where the central fibre has only isolated singularities. 

\smallskip

Now, we consider a reasonable "good" family of Calabi--Yau varieties and ask how the bound on the Chern-Ricci potentials vary in families.
Assume that $\CX$ is a normal variety, $K_\CX$ is trivial and $\pi: \CX \ra \BD$ is a smoothing.
Moreover, we suppose that $X_0$ has only isolated canonical singularities.
One can find a trivializing section $\Om$ of $K_{\CX/\BD}$. 
The restriction on each fibre $\Om_t := \res{\Om}{X_t}$ defines a trivialization of $K_{X_t}$.
Following from \cite[Theorem E]{Guedj_Lu_3_2021}, for each $t$, there is a bounded solution to the corresponding complex Monge--Amp\`ere equation of canonical density.
Then we show a uniform estimate in families:

\begin{bigthm}\label{bigthm:CY_family}
	Suppose that $\CX$ is normal, $K_{\CX}$ is trivial, and $\pi: \CX \ra \BD$ is a smoothing of a variety $X_0$ whose singularities are canonical and isolated. 
	For each $t \in \BD$, let $(\vph_t,c_t) \in \lt( \PSH(X_t, \om_t) \cap L^\infty(X_t) \rt) \times \BR_{>0}$ be a pair solving the complex Monge--Amp\`ere equation
	\[
	(\om_t + \ddc_t \vph_t)^n = c_t \Om_t \w \overline{\Om_t}
	\quad\text{and}\quad
	\sup_{X_t} \vph_t = 0.
	\]
	Then there is a uniform constant $C_{\MA}$ such that for all $t \in \BD_{1/2}$
	\[
	c_t + c_t^{-1} + \norm{\vph_t}_{L^\infty} \leq C_{\MA}.
	\]
\end{bigthm}

\subsection*{Structure of the article}
The paper is organized as follows:
\begin{itemize}
	\item In Section \ref{sec:preliminaries}, we recall basic notions in pluripotential theory and singular spaces.
	\item Section \ref{sec:chasing_the_constants} is a recap on methods to obtain $L^\infty$-estimates in local and global cases.
	\item In Section~\ref{sec:unif_Skoda}, we study the local and global uniform Skoda estimates.
	\item In Section \ref{sec:vol_cap}, we deal with the volume-capacity comparison stated in Section \ref{sec:chasing_the_constants}.
	\item In Section \ref{sec:L1_estimate}, we establish Conjecture~\ref{conj:SL} under the geometric assumption~\ref{geoasm:L1_assumption}. 
	\item In Section \ref{sec:geometric_applications}, we focus on families of Calabi--Yau varieties and show Theorem~\ref{bigthm:CY_family}.
\end{itemize}

\begin{ack*}
The author is grateful to his thesis advisors Vincent Guedj and Henri Guenancia for their constant supports, suggestions, and encouragements.
The author thanks Ahmed Zeriahi for many helpful discussions.
The author would like to thank the anonymous referee for useful suggestions which helped to improve the exposition.
The author is partially supported by ANR-11-LABX-0040 (research project HERMETIC) and EUR MINT project ANR-18-EURE-0023.
\end{ack*}

\section{Preliminaries}\label{sec:preliminaries}
In this section, we recall some definitions, notations, and conventions which will be used in the sequel.
We define the twisted exterior derivative by $\dc = \frac{\ii}{2}(\db - \pl)$ and we then have $\ddc = \ii\ddb$.
We denote by
\begin{itemize}
	\item $\BD_r := \set{z \in \BC}{|z|<r}$ the open disk of radius $r$;
	\item $\BD^\ast_{r} := \set{z \in \BC}{0 < |z| < r}$ the punctured disk of radius $r$.
\end{itemize}
When $r=1$, we simply write $\BD := \BD_1$ and $\BD^\ast := \BD^\ast_1$.

\subsection{Smooth forms and currents on singular spaces} 
Let $X$ be a reduced complex analytic space of pure dimension $n \geq 1$.
We denote by $X^\reg$ the complex manifold of regular points of $X$ and $X^\sing := X \setminus X^\reg$ the singular set of $X$.
Now, we recall definitions of smooth forms and currents on a complex analytic space $X$:
\begin{defn}\label{defn:metrics_and_forms}
	We say that
	\begin{enumerate}
		\item
		A smooth form $\af$ on $X$ is the data of a smooth form on $X^\reg$ such that given any local embedding $X \xhookrightarrow[\loc.]{} \BC^N$, $\af$ extends smoothly to $\BC^N$;
		\item
		A smooth hermitian metric $\om$ on $X$ is a smooth $(1,1)$-form which locally extends to a hermitian metric on $\BC^N$;
		\item
		$\CD_{p,q}(X)$ (resp. $\CD_{p,p}(X)_\BR$) is the space of compactly supported (resp. real) smooth forms of bidegree $(p,q)$;
		\item
		The notion of currents, $\CD'_{p,q}(X)$ (resp. $\CD'_{p,p}(X)_\BR$), is defined by acting on (resp. real) smooth forms with compact support.
	\end{enumerate}
\end{defn} 
The operators $\dd$, $\dc$ and $\ddc$ are well-defined by duality (see \cite{Demailly_1985} for details).

\subsection{Plurisubharmonic functions}
Let $\Om$ is an open domain in $\BC^n$. 
We say that $u$ is a \textit{plurisubharmonic} function (\textit{psh} for short) on $\Om$ if it is upper semicontinuous and satisfies the sub-mean inequality on each complex line through every point $x \in \Om$:
\[
u(x) \leq \frac{1}{2\pi} \int_{0}^{2\pi} u(x + \zt \e^{\ii \ta}) \dd \ta,
\quad \forall x \in \Om \text{ and } \forall \zt \in \BC^n \text{ such that } |\zt| < \dist(x, \pl \Om).
\]
We denote by $\PSH(\Om)$ the space of all psh functions on $\Om$.

Suppose that $X$ is a reduced complex analytic space equipped with a hermitian metric $\om$.
\begin{defn}
	Let $u: X \ra [-\infty, \infty)$ be a given function.
	We say that
	\begin{enumerate}
		\item 
		$u$ is a psh function on $X$ if it is locally the restriction of a psh function under local embeddings of $X$ into $\BC^N$;
		\item
		$u$ is \textit{quasi-plurisubharmonic} (\textit{qpsh} for short) on $X$ if it can be locally written as the sum of a psh and a smooth function;
		\item
		$\PSH(X,\om)$ is the set of all $\om$-\textit{plurisubharmonic} (abbreviated to $\om$-\textit{psh}), namely, the set of all qpsh functions $u$ which satisfies $\om + \ddc u \geq 0$ in the sense of currents.
	\end{enumerate}
\end{defn}

\begin{rmk}
There is a weaker notion of (quasi-)plurisubharmonic functions. 
We say that $u$ is weakly (quasi-)plurisubharmonic on $X$ if $u$ is locally bounded from above on a variety $X$ and its restriction to the complex manifold $X^\reg$ is (quasi-)plurisubharmonic. 
On a locally irreducible variety, the stronger notion given above and the weaker notion are equivalent (cf. \cite[Th\'eor\`eme~1.7]{Demailly_1985}).
In this article, we assume that $X$ also is locally irreducible in some places in order to make sense of the envelope constructions that might not be (quasi-)plurisubharmonic (in the strong sense) on locally reducible complex spaces. 
\end{rmk}

\subsection{Lelong numbers}
Lelong numbers describe the local behavior of currents or psh functions near a point at which it has a log pole.
Here we recall a generalized definition given by Demailly:
\begin{defn}[{\cite[D\'efinition 3]{Demailly_1982}}]\label{defn:Lelong_number}
	Let $X$ be a complex analytic space.
	If $T$ is a closed positive $(p,p)$-current on $X$ and if $x \in X$ is a fixed point, then the Lelong number of $T$ at $x$ is defined as the decreasing limit 
	\[
	\nu(T,x) := \lim_{r \ra 0} \frac{1}{r^{2(n-p)}} \int_{\{\psi < r\}} T \w (\ddc \psi)^{n-p}
	= \int_{\{x\}} T \w (\ddc \log \psi)^{n-p}
	\]
	where $\psi = \sum_{i \in I} |g_i|^2$ and $(g_i)_{i \in I}$ is any finite system of generators of the maximum ideal $\mathfrak{m}_{X,x} \subset \CO_{X,x}$.
\end{defn}


\subsection{Monge--Amp\`ere capacities}
The notion of Monge--Amp\`ere capacities was given by Bedford and Taylor \cite{Bedford_Taylor_1982}.
Using Monge--Amp\`ere capacities, they proved that the negligible sets are pluripolar. 
\begin{defn}
	Let $E \subset \Om$ be a Borel subset. 
	The Bedford--Taylor capacity (or Monge--Amp\`ere capacity) is defined by
	\[
	\CAP(E;\Om) 
	:= \sup \set{\int_E (\ddc u)^n}{ u \in \PSH(\Om) \text{ and } -1 \leq u \leq 0}.
	\]
\end{defn}

\begin{thm}[\cite{Bedford_Taylor_1982}]
	A subset $E \subset \Om$ is pluripolar if and only if $\CAP(E;\Om) = 0$
\end{thm}

For global versions, Ko{\l}odziej \cite{Kolodziej_2003} first defined the Monge--Amp\`ere capacity on a given compact K\"ahler manifold $(X,\om)$.
The definition is analogous to the local cases:
\begin{defn}
	Let $E \subset X$ be a Borel subset.
	Define
	\[
	\CAP_\om(E) := \set{\int_E (\om + \ddc u)^n}{u \in \PSH(X,\om) \text{ and } -1\leq u \leq 0}.
	\]
	These definitions can also be extended to non-closed or non-positive form $\om$ and singular complex analytic space $X$ (cf. \cite{Demailly_1985, EGZ_2009, GZbook, Guedj_Guenancia_Zeriahi_2020}).
\end{defn}

\section{Uniform $L^\infty$-estimate}\label{sec:chasing_the_constants}
In this section, we mainly pay attention to $L^\infty$-estimates of complex Monge--Amp\`ere equations on pseudoconvex domains and compact hermitian varieties.
We shall follow the method given by Guedj--Ko{\l}odziej--Zeriahi \cite{Guedj_Kolodziej_Zeriahi_2008} and Guedj--Lu \cite{Guedj_Lu_3_2021} to produce a priori estimates. 
We also compute the precise dependence of these $L^\infty$-estimates on background data.

\smallskip

\subsection{Local $L^\infty$-estimate}\label{subsec:Kolodziej}
In this section, our goal is to establish a refined version of Ko{\l}odziej's a priori estimate \cite{Kolodziej_1998} of complex Monge--Amp\`ere equation on singular strongly pseudoconvex domain.
We recall the definition of a strongly pseudoconvex domain on a Stein space as in \cite[Section 1]{Guedj_Guenancia_Zeriahi_2020}.
Let $S$ be a singular Stein space which is reduced and locally irreducible, of complex dimension $n \geq 1$.
There is a proper embedding $S \hookrightarrow \BC^N$ for some $N$ large.
A domain $\Om \Subset S$ is strongly pseudoconvex if it admits a negative smooth psh exhaustion, i.e. a function $\rho$ smooth strongly psh in a neighborhood $\Om'$ of $\overline{\Om}$ such that $\Om := \set{x \in \Om'}{\rho(x) < 0}$.
Fix a hermitian metric $\bt$ on $\BC^N$ and define a volume form on $S$ by taking $\dd V = \res{\bt^n}{S}$.

\smallskip

First, we note that the following estimate always holds:
\begin{taggedvolcap}{(VC)}\label{bigasm:VC_comparison}
	For every $k>1$, there exists a constant $C_{VC,k}>0$ such that 
	\[
	\forall K \Subset \Om, \quad
	\Vol(K) \leq C_{VC,k} \CAP^k (K; \Om).
	\]
\end{taggedvolcap}
The proof of the volume-capacity comparison will be given in Section~\ref{sec:vol_cap} not only in a fixed pseudoconvex set $\Om$ but also for families.

\smallskip

Fix a density $0 \leq f \in L^p(\Om,\dd V)$.
Suppose that $\vph \in \PSH(\Om) \cap L^\infty(\Om)$ is the solution to the following Dirichlet problem of complex Monge--Amp\`ere equation
\begin{equation}\tag{locMA}\label{eq:local_MA_Kolo}
	\begin{cases}
		(\ddc \vph)^n = f \dd V & \text{in }\Om,\\
		\vph = 0 & \text{on }\pl\Om.
	\end{cases}    
\end{equation}
In smooth setting, the existence, uniqueness and the $L^\infty$-estimate of the continuous solution of (\ref{eq:local_MA_Kolo}) has been constructed by Ko{\l}odziej~\cite{Kolodziej_1998}. 
The existence and uniqueness have been extended by Guedj--Guenancia--Zeriahi~\cite[Theorem A]{Guedj_Guenancia_Zeriahi_2020} to singular contexts.
Now, for singular setup, we prove the following refined version of Ko{\l}odziej's $L^\infty$-estimate: 

\begin{thm}[Ko{\l}odziej's $L^\infty$-estimate]\label{thm:Kolodziej}
	Fix $0 \leq f \in L^p(\Om, \dd V)$ with $p > 1$. 
	Suppose that $\vph \in \PSH(\Om) \cap L^\infty(\Om)$ is the solution to the complex Monge--Amp\`ere equation (\ref{eq:local_MA_Kolo}).
	Then 
	\[
	\norm{\vph}_{L^\infty} \leq C_{\Kol,p} \norm{f}_{L^p}^{\frac{1}{n}}
	\quad\text{where}\quad 
	C_{\Kol,p} :=
	\lt[1 + \lt(\frac{\e}{1-\e^{-1}}\rt) C^{1/qn}_{VC,2q}
	2^{1+ \frac{1}{q}}
	\Vol^{\frac{1}{nq^2}}(\Om)
	\lt(n! C_\rho \norm{\rho}_{L^\infty}^n\rt)^{\frac{1}{nq}}\rt],
	\]
	$1/p+ 1/q = 1$, $\dd V \leq C_\rho (\ddc \rho)^n$, and $C_{VC,2q} > 0$ is a constant in the volume-capacity comparison \ref{bigasm:VC_comparison} such that $\Vol(K) \leq C_{VC,2q} \CAP^{2q}(K;\Om)$ for all $K \Subset \Om$.
\end{thm}

\begin{proof}
	The idea of proof comes from the work of Guedj--Ko{\l}odziej--Zeriahi~\cite[Section 1]{Guedj_Kolodziej_Zeriahi_2008}.
	We are going to prove the following statement: given $\vep>0$, we have $\norm{\vph}_{L^\infty(\Om)} \leq M_\vep$ where 
	\[
	M_\vep = \vep + \lt(\frac{\e}{1-\e^{-1}}\rt) C^{1/qn}_{VC,2q}
	\lt(\frac{2}{\vep}\rt)^{1+ \frac{1}{q}}
	\norm{f}^{\frac{2}{n}+\frac{1}{nq}}_{L^p} \Vol^{\frac{1}{nq^2}}(\Om)
	\lt(n! C_\rho \norm{\rho}_{L^\infty}^n\rt)^{\frac{1}{nq}}.
	\]
	In particular, when $\vep = \norm{f}_{L^p}^{\frac{1}{n}}$, one get the desired estimate.
	Before explaining the proof, we recall some useful facts.
	For simplicity, we denote by $\CAP(\bullet) = \CAP(\bullet;\Om)$.
	First, we recall some basic lemmas:
	
	\begin{lem}\label{lem:lem_from_comparison}
		Fix $\vph, \psi \in \PSH(\Om) \cap L^{\infty}(\Om)$ such that $\liminf_{z \ra \pl\Om} (\vph - \psi) \geq 0$.
		Then for all $t,s > 0$ we have
		\[
		t^n \CAP(\{\vph - \psi < -s-t\}) 
		\leq \int_{\{\vph -\psi < -s\}} (\ddc \vph)^n. 
		\]
	\end{lem}
	
	By definition, for all $u \in \PSH(\Om) \cap L^\infty(\Om)$, complex Monge--Amp\`ere measures $(\ddc u)^n$ put zero mass on $\Om^\sing$; hence we still have the comparison principle on $\Om$ (cf. \cite[Proposition 1.5]{Guedj_Guenancia_Zeriahi_2020}). 
	Then one can follow exactly the same argument in \cite[Lemma 1.3]{Guedj_Kolodziej_Zeriahi_2008} to obtain Lemma~\ref{lem:lem_from_comparison}.
	
	\smallskip
	
	Using the volume-capacity comparison~\ref{bigasm:VC_comparison} and H\"older's inequality, one has the estimate as follows
	\begin{lem}\label{lem: Lp cap estimate}
		For all $\tau > 1$, there exists a constant $D_\tau := C^{\frac{p-1}{p}}_{VC,k} \norm{f}_{L^p}$ where 
		$k = k(\tau,p) := \frac{\tau p}{(p-1)} = \tau q$
		such that
		\[
		\forall K \Subset \Om, \quad
		0 \leq \int_K f \dd V \leq D_\tau \CAP^\tau(K).
		\]
	\end{lem} 
	
	The following classical lemma is due to De Giorgi and the reader is referred to \cite[Lemma 1.5]{Guedj_Kolodziej_Zeriahi_2008} and \cite[Lemma 2.4]{EGZ_2009} for the proof.
	\begin{lem}\label{lem:DeGirogi_lem}
		Let $g: \BR_{\geq 0} \ra \BR_{\geq 0}$ be a decreasing right-continuous function and satisfy $\lim_{s \ra \infty} g(s) = 0$.
		Assume that there exists $\tau > 1, B > 0$ such that $g$ satisfies 
		\[
		\forall s, t > 0,\quad t g(s+t) \leq B (g(s))^\tau.
		\]
		Then $g(s)=0$ for all $s \geq s_\infty$,
		where 
		\[
		s_\infty = \frac{\e B(g(0))^{\tau-1}}{1 - \e^{1-\tau}}.
		\]
	\end{lem}
	
	By Lemma~\ref{lem:lem_from_comparison} and Lemma~\ref{lem: Lp cap estimate}, we have
	\[
	t^n\CAP(\{\vph < -s-t\}) 
	\leq \int_{\{\vph < -s\}} (\ddc\vph)^n
	= \int_{\{\vph < -s\}} f \dd V
	\leq D_2 \CAP^2(\{\vph < -s\})
	\]
	where $D_2 = C_{VC,2q}^{1/q} \norm{f}_{L^p}$.
	Thus,
	\[
	t \CAP^{1/n}(\{\vph<-s-t\})
	\leq D_2^{1/n} \CAP^{2/n}(\{\vph<-s\}).
	\]
	Let $g(s) := \CAP^{1/n}(\{\vph<-s-\vep\})$.
	Then we have $t g(s+t) \leq B g(s)^2$ where $B = D_2^{1/n}$.
	Using Lemma~\ref{lem:DeGirogi_lem}, one obtains $\CAP(\{\vph < -s -\vep\}) = 0$ for all $s \geq s_\infty = \frac{\e B g(0)}{1-\e^{-1}}$.
	This implies that $\vph \geq -s_\infty -\vep$ almost everywhere and thus everywhere by plurisubharmonicity. 
	Therefore, one has
	\[
	\sup_\Om (-\vph) 
	\leq \vep + s_\infty 
	= \vep + \frac{\e B g(0)}{1-\e^{-1}}.
	\]
	Now, we need to control $g(0) = \CAP^{1/n}(\{\vph<-\vep\})$.
	By Lemma~\ref{lem:lem_from_comparison} and Chebyshev inequality for a fixed constant $r > 0$, we have
	\begin{align*}
		\lt(\frac{\vep}{2}\rt)^n \CAP\lt(\lt\{\vph < -\frac{\vep}{2}-\frac{\vep}{2}\rt\}\rt) 
		&\leq \int_{\{\vph<-\vep/2\}} (\ddc\vph)^n
		= \int_{\{\vph<-\vep/2\}} f \dd V\\
		&\leq \int_\Om \lt(\frac{-2\vph}{\vep}\rt)^r f \dd V 
		\leq \norm{f}_{L^p} \lt(\int_\Om \lt(\frac{-2\vph}{\vep}\rt)^{rq} \dd V\rt)^{1/q}.
	\end{align*}
	Put $r=n/q$. 
	Note that integration by parts is legitimate in our setting (see \cite[Lemma 2.11]{DGG2020}).
	By B{\l}ocki's estimate of integration by parts \cite[Theorem 2.1]{Blocki_1993} and H\"older inequality, 
	one can derive that
	\begin{align*}
		\int_\Om (-\vph)^n \dd V 
		&\leq C_\rho \int_\Om (-\vph)^n (\ddc \rho)^n \\
		&\leq n! C_\rho \norm{\rho}_{L^\infty}^n \int_{\Om} (\ddc\vph)^n
		\leq n! C_\rho \norm{\rho}_{L^\infty}^n \Vol^{\frac{1}{q}}(\Om) \norm{f}_{L^p}.
	\end{align*}
	One can infer 
	\[
	g(0)^n = \CAP(\{\vph < -\vep\}) 
	\leq \lt(\frac{2}{\vep}\rt)^{n+\frac{n}{q}} \norm{f}^{1+\frac{1}{q}}_{L^p} 
	\Vol^{\frac{1}{q^2}}(\Om)
	\lt(n! C_\rho \norm{\rho}_{L^\infty}^n\rt)^{1/q}.
	\]
	All in all, we obtain the desired estimate: 
	\[
	\norm{\vph}_{L^\infty} 
	\leq \vep + \lt(\frac{\e}{1-\e^{-1}}\rt) C^{1/qn}_{VC,2q}
	\lt(\frac{2}{\vep}\rt)^{1+ \frac{1}{q}}
	\norm{f}^{\frac{2}{n}+\frac{1}{nq}}_{L^p} \Vol^{\frac{1}{nq^2}}(\Om)
	\lt(n! C_\rho \norm{\rho}_{L^\infty}^n\rt)^{\frac{1}{nq}}.
	\] 
\end{proof}

\smallskip

\subsection{Global $L^\infty$-estimate}
Suppose that $(X,\om)$ is an $n$-dimensional compact locally irreducible hermitian variety.
Fix a function $0 \leq f \in L^p(X,\om^n)$ with $p>1$.
First, we fix some notations:
\begin{enumerate}
	\item
	Denote the volume of $X$ with respect to $\om^n$ by $V$;
	
	\item
	A constant $B' > 0$ is such that $-B' \om^n \leq \ddc \om^{n-1} \leq B' \om^n$;
	
	\item 
	Fix a 
	finite double cover $(\Om_j := \{\rho_j < 0\})_{1 \leq j \leq N}$ and $(\Om_j' := \{\rho_j < -c_j\})_{1 \leq j \leq N}$ of $X$ where for each $j$, the function $\rho_j$ is smooth on $X$, strictly psh near $\overline{\Om}_j$, and $0 \leq \rho_j \leq 1$ on $X \setminus \Om_j$, and $c_j > 0$ is a constant;
	
	\item 
	$C_{\Kol,p}$ is a constant such that $\norm{u_j}_{L^\infty(\Om_j)} \leq C_{\Kol,p} \norm{f}_{L^p(\Om_j,\om^n)}^{\frac{1}{n}}$ where for each $1 \leq j \leq N$, the function $u_j$ is the solution to the Dirichlet problem
	\[
	\begin{cases}
		(\ddc u_j)^n = f \om^n & \text{in }\Om_j,\\
		u_j = 0 & \text{on }\pl\Om_j.
	\end{cases}    
	\]
	
	\item $A_\rho$ is a constant such that $A_\rho \om + \ddc \rho_j > 0$ on $X$ for all $1 \leq j \leq N$;
	
	\item $c_\rho = \min_{1 \leq j \leq N} c_j > 0$.
\end{enumerate}

\smallskip

In this section, we fix some geometric constants and impose two integral bounds on the density $f$:

\begin{taggedgeoconst}{(SL)}\label{bigasm:SL}
	There is a constant $C_{SL} > 0$ such that the following inequality holds
	\[
	\forall \vph \in \PSH(X,\om), \quad
	\sup_X \vph - C_{SL} \leq \frac{1}{V} \int_X \vph \om^n \leq \sup_X \vph;
	\]
\end{taggedgeoconst}

\begin{taggedgeoconst}{(Skoda)}\label{bigasm:Skoda}
	There exist $\af > 0$ and $A_\af > 0$ such that
	\[
	\forall u \in \PSH(X,\om), \quad \int_X \e^{\af \lt( \sup_X u - u \rt)} \om^n \leq A_\af.
	\]
\end{taggedgeoconst}

\begin{taggedanaconst}{(AC)}\label{bigset:AC}
	Let $0 \leq f \in L^p(X,\om)$ for some $p > 1$. 
	There are two constants $c_f, C_f > 0$ such that 
	\[
	c_f \leq \int_X f^\frac{1}{n} \om^n
	\quad\text{and}\quad
	\norm{f}_{L^p} \leq C_f.
	\]
\end{taggedanaconst}

Following the strategy in \cite{Guedj_Lu_3_2021}, we shall prove an a priori $L^\infty$-estimate of complex Monge--Amp\`ere equations on hermitian varieties. 
\begin{thm}
	Let $(\vph, c) \in \lt(\PSH(X,\om) \cap L^\infty(X)\rt) \times \BR_{>0}$ be a pair solving the complex Monge--Amp\`ere equation
	\begin{equation}\tag{MA}\label{eq:DCMAE}
		(\om + \ddc \vph)^n = c f \om^n 
		\quad \text{and} \quad
		\sup_X \vph = 0.
	\end{equation}
	There exists a uniform positive constant $C_{\MA}>0$ such that
	\[
	c + c^{-1} + \norm{\vph}_{L^\infty(X)} \leq C_{\MA}
	\]
	and the constant $C_{\MA}$ depends only on $n$, $V$, $B'$, $N$, $A_\rho$, $c_\rho$, $\af$, $A_\af$, $C_{\Kol}$, $p$, $c_f$, $C_f$, and $C_{SL}$.
\end{thm}

{
	\begin{rmk}
		Suppose that $X$ is a compact normal variety and $\mu: \wX \to X$ is a resolution of singularities.
		Following \cite[Theorem B]{Guedj_Lu_3_2021}, there exists a pair $(\vph, c)$ solving the corresponding complex Monge--Amp\`ere equation on $\wX$.
		Let $E = \Exc(\mu)$ be the exceptional divisor of $\mu$ which is analytic and thus pluripolar. 
		Since the complex Monge--Amp\`ere measure $(\mu^\ast \om + \ddc \vph)^n$ charges no mass on $E$, one can descend the solution to $X^\reg$. 
		By the extension theorem of Grauert and Remmert \cite[Satz 4]{Grauert_Remmert_1956} and normality of $X$, the $\mu_\ast \vph$ induces a $\om$-psh function on $X$ and it solves the complex Monge--Amp\`ere equation.
	\end{rmk}
}

\subsubsection{Upper bound of $c$}
Following the same idea as in \cite[Lemma 5.9]{Kolodziej_Nguyen_2015}, one can find an upper bound of $c$ simply using the arithmetic-geometric mean inequality:
\begin{lem}\label{lem:upper_bound_c}
	With the geometric constant $C_{SL}$ in \ref{bigasm:SL}, one has 
	\[
	c \leq \lt(\frac{C_\Lap}{n c_f}\rt)^n
	\quad\text{where } C_\Lap := V (1 + B' C_{SL}).
	\]
\end{lem}

\begin{proof}
	For all $u \in \PSH(X,\om)$, we compute
	\[
	\int_X (\om + \ddc u) \w \om^{n-1} 
	= \int_X \om^n + \int_X (u - \sup_X u) \ddc \om^{n-1} 
	\leq V(1 + B' C_{SL}) =: C_{\Lap}.
	\]
	Then applying arithmetic-geometric mean inequality, one can infer that 
	\[
	c^{1/n} \int_X f^{1/n} \om^n 
	= \int_X \lt(\frac{(\om + \ddc \vph)^n}{\om^n}\rt)^{\frac{1}{n}} \om^n
	\leq \frac{1}{n} \int_X (\om + \ddc \vph) \w \om^{n-1}
	\leq \frac{1}{n} C_{\Lap}.
	\]
	Rearranging the inequality, we obtain an upper bound of $c$ as desired.
\end{proof}

\smallskip

\subsubsection{Domination principle}
The domination principle has been proved under several setup (cf. \cite{Nguyen_2016, Lu_Phung_To_2021, Guedj_Lu_2_2021, Guedj_Lu_3_2021} and references therein).
{
	Here we establish the following domination principle on singular varieties:
}

\begin{lem}[Domination principle]\label{lem:domination_principle}
	Let $u, v \in \PSH(X,\om) \cap L^\infty(X)$.
	Then we have the following properties
	\begin{enumerate}
		\item if $(\om + \ddc u)^n \leq c (\om + \ddc v)^n$, then $c \geq 1$;
		\item if $\e^{-\ld u} (\om + \ddc u)^n \leq \e^{-\ld v}(\om + \ddc v)^n$ for some $\ld > 0$, then $v \leq u$.
	\end{enumerate}
\end{lem}

{
	\begin{proof}[Sketch of proof]
		Let $\mu: \wX \to X$ be a resolution of singularities. 
		The $(1,1)$-form $\mu^\ast \om$ is semi-positive and big (cf. \cite[Definition 1.6]{Guedj_Lu_3_2021}).
		For all functions $u \in \PSH(X,\om) \cap L^\infty(X)$, the Monge--Amp\`ere measure $(\mu^\ast\om + \ddc \mu^\ast u)^n$ puts no mass on the exceptional divisor $E := \Exc(\mu)$ since $E$ is a pluripolar set.
		On the other hand, by definition, $(\om + \ddc u)^n$ charges no mass on $X^\sing$ for all $u \in \PSH(X,\om) \cap L^\infty(X)$. 
		Hence, one can descend the domination principles in \cite[Corollary 1.13 and 1.14]{Guedj_Lu_3_2021} to $X$.
		Then one can conclude Lemma~\ref{lem:domination_principle}.
	\end{proof}
}

\smallskip

\subsubsection{Subsolution estimate}

In \cite[Theorem 2.1]{Guedj_Lu_3_2021}, Guedj and Lu constructed a subsolution $(\psi, m)$ to the complex Monge--Amp\`ere equation with a given $L^p$-density $g$.
We follow the same method to get subsolution estimate on singular varieties while also carefully keeping track of the dependence of data.

\begin{prop}\label{prop:subsolution}
	For all $p > 1$, there exist uniform constants $m_p, M_p > 0$ such that for every $0 \leq g \in L^p(X,\om)$ with $\norm{g}_{L^p} = 1$, there is a function $\psi \in \PSH(X,\om) \cap L^\infty(X)$ satisfying
	\[
	(\om + \ddc \psi)^n \geq m_p g \om^n
	\quad\text{and}\quad
	\osc_X \psi \leq M_p.
	\]
	Precisely, the constants $m_p$ and $M_p$ can be taken as
	\[
	m_p = \frac{c_\rho}{N A_\rho \lt(C_{\Kol,p} + 1\rt)}
	\quad\text{and}\quad
	M_p = C_{\Kol,p} \lt(\frac{1}{A_\rho} + 1\rt).
	\]
\end{prop}

\begin{proof}
	For each $j$, let the function $u_j \in \PSH(\Om_j) \cap L^\infty(\Om_j)$ be the unique solution to the following complex Monge--Amp\`ere equation
	\[
	\begin{cases}
		(\ddc u_j)^n = g \om^n & \text{in }\Om_j,\\
		u_j = -1 & \text{on } \pl\Om_j.
	\end{cases}
	\]
	From Theorem~\ref{thm:Kolodziej}, there is a constant $C_{\Kol,p} > 0$ such that for all $j$, 
	\[
	\norm{u_j}_{L^\infty} 
	\leq C_{\Kol,p} \norm{g}_{L^p(\Om_j, \om^n)}^{1/n} + 1 
	\leq C_{\Kol,p} + 1.
	\]
	Now, we consider the psh functions $(v_j)_j$ defined by $v_j := \max\lt\{u_j, \frac{C_{\Kol,p} + 1}{c_j} \rho_j \rt\}$ for each $j$.
	One can see that the following properties are satisfied
	\begin{itemize}
		\item $v_j = u_j$ in $\Om_j'$ and $(\ddc v_j)^n = g\om^n$ in $\Om_j'$;
		\item $v_j = \frac{C_{\Kol,p} + 1}{c_j} \rho_j$ on $X \setminus \Om_j$ and near the boundary of $\pl\Om_j$.
	\end{itemize}
	From the construction, one can check that for every $j$, $\lt(\frac{A_\rho \lt(C_{\Kol,p} + 1\rt)}{c_\rho}\rt) \om + \ddc v_j \geq 0$ as well.
	We define the subsolution $\psi$ as follows
	\[
	\psi = \frac{c_\rho}{N A_\rho \lt(C_{\Kol,p} + 1\rt)}\sum_{j=1}^N v_j.
	\]
	Note that in $\Om_j'$, we have 
	\begin{align*}
		(\om + \ddc \psi)^n 
		&= \lt(\frac{c_\rho}{N A_\rho \lt(C_{\Kol,p} + 1\rt)} \sum_{j=1}^N \lt[\frac{A_\rho \lt(C_{\Kol,p} + 1\rt)}{c_\rho} \om + \ddc v_j \rt] \rt)^n\\
		&\geq \frac{c_\rho}{N A_\rho \lt(C_{\Kol,p} + 1\rt)} (\ddc u_j)^n 
		= \frac{c_\rho}{N A_\rho \lt(C_{\Kol,p} + 1\rt)} g \om^n.
	\end{align*}
	Then we derive that
	\[
	\psi \leq \frac{1}{A_\rho}
	\quad\text{and}\quad
	\psi \geq -C_{\Kol,p} - 1
	\implies
	\osc_X \psi 
	\leq \frac{1}{A_\rho} + C_{\Kol,p} + 1.
	\]
\end{proof}

\subsubsection{$L^\infty$-estimate}
We now prove an a priori $L^\infty$-estimate following the method in \cite[Theorem 2.1]{Guedj_Lu_3_2021}.

\begin{thm}
	Let $(X,\om)$ be a compact hermitian variety with $\dim_\BC X = n$.
	Fix a density $0 \leq f \in L^p(X,\om^n)$.
	Assume that $c_f \leq \int_X f^{\frac{1}{n}} \om^n$ for a constant $c_f > 0$.
	Let the pair $(\vph, c) \in \lt(\PSH(X,\om) \cap L^\infty(X)\rt) \times \BR_{>0}$ be a solution to (\ref{eq:DCMAE}).
	Then one has
	{\small
		\[
		c \geq \frac{m_{\frac{p+1}{2}}}{A_\af^{\frac{p-1}{p(p+1)}} C_f}, \quad
		c \leq \lt(\frac{C_{\Lap}}{n c_f}\rt)^n, \quad \text{and} \quad
		\norm{\vph}_{L^\infty} 
		\leq M_{\frac{p+1}{2}} + \frac{2 p (p+1)}{\af (p-1)} \lt( \log \frac{A_\af^{\frac{p-1}{p(p+1)}} C_f C_\Lap^n}{m (n c_f)^n} \rt).
		\]}

\end{thm}

\begin{proof}
	We consider a twisted function $g' = \e^{-\vep \vph}f$ for some $\vep > 0$.
	Fix constants $p' = \frac{p+1}{2} \in (1,p)$ and $\vep = \frac{\af(p-1)}{2 p (p+1)} = \frac{\af (p-p')}{2 p p'} \in \lt(0, \frac{\af (p-p')}{p p'}\rt)$. 
	One can derive that $g' \in L^{p'}$.
	Indeed, by H\"older inequality, we have
	\[
	\norm{g'}_{L^{p'}} 
	\leq 
	\norm{\e^{-\vep \vph}}_{L^{\frac{pp'}{p-p'}}} \norm{f}_{L^p}
	\leq A_\af^{\frac{p-p'}{pp'}} \norm{f}_{L^p}.
	\]
	Put $g = g'/ \norm{g'}_{L^{p'}}$.
	From Proposition~\ref{prop:subsolution}, we have a bounded $\om$-psh function $\psi$ with $\sup_X \psi = 0$ such that
	\begin{align*}
		(\om + \ddc \psi)^n 
		&\geq m_{p'} g \om^n 
		= m_{p'} \frac{\e^{-\vep \vph} f}{\norm{g'}_{L^{p'}}} \om^n 
		= \frac{m_{p'} \e^{-\vep \vph}}{c \norm{g'}_{L^{p'}}} c f \om^n
		= \frac{m_{p'} \e^{-\vep \vph}}{c \norm{g'}_{L^{p'}}} (\om + \ddc \vph)^n\\
		&\geq \frac{m_{p'}}{c \norm{g'}_{L^{p'}}} (\om + \ddc \vph)^n.
	\end{align*}
	and $\norm{\psi}_{L^\infty} \leq M_{\frac{p+1}{2}}$.
	By Lemma~\ref{lem:domination_principle}, one get $\frac{m_{p'}}{c \norm{g'}_{L^{p'}}} \leq 1$; hence $c$ has a lower bound, $c \geq \frac{m_{p'}}{A_\af^{\frac{p-p'}{pp'}} \norm{f}_{L^p}}$.
	Also, we see that
	\[
	\e^{-\vep \psi} (\om + \ddc \psi) 
	\geq \frac{m_{p'} \e^{-\vep\vph}}{c \norm{g'}_{L^{p'}}} (\om + \ddc \vph)^n 
	= \exp\lt(-\vep\lt(\vph - \frac{1}{\vep} \log \frac{m_{p'}}{c \norm{g'}_{L^{p'}}} \rt) \rt) (\om + \ddc \vph)^n.
	\]
	Applying the domination principle again, one can infer
	\[
	M_{p'} 
	\leq \psi 
	\leq \vph - \frac{1}{\vep} \lt( \log \frac{m_{p'}}{c \norm{g'}_{L^{p'}}} \rt)
	\]
	Finally, this provides a uniform $L^\infty$-estimate as follows 
	\begin{align*}
		\norm{\vph}_{L^\infty} 
		&\leq M_{p'} - \frac{1}{\vep} \lt(\log \frac{m}{c \norm{g'}_{L^{p'}}} \rt)
		\leq M_{p'} + \frac{1}{\vep} \lt( \log \frac{A_\af^{\frac{p-p'}{pp'}} \norm{f}_{L^p} C_\Lap^n}{m (n c_f)^n} \rt).
	\end{align*} 
\end{proof}

\subsection{Proof of Theorem~\ref{bigthm:uniform_a_priori_estimate}}
Recall that the constant $C_{\MA}$ depends only on $n$, $V$, $B'$, $N$, $A_\rho$, $c_\rho$, $\af$, $A_\af$, $C_{\Kol}$, $p$, $c_f$, $C_f$, and $C_{SL}$.
We shall control these constant to get Theorem~\ref{bigthm:uniform_a_priori_estimate}.
The following data are included in the assumption of Theorem~\ref{bigthm:uniform_a_priori_estimate}:
\begin{itemize}
	\item $n$ is fixed in the geometric setting \ref{geoset:irr_family};
	\item $p$, $c_f$, $C_f$ are data in the integral bounds (\ref{integral_bound:IB});
	\item $C_{SL}$ is given by Conjecture \ref{conj:SL}.
\end{itemize}
Then fixing some choices of background data, we have uniform control of the following constants:
\begin{itemize}
	\item $A_\rho, c_\rho, N$: After shrinking $\BD$, we can cover $\CX$ by finitely many pseudoconvex double cover $(\CU_j = \{\rho_j < 0\})_j$ and $(\CU_j' = \{\rho_j < -c_j\})_j$.
	Then the slices $\Om_j := \CU_j \cap X_t = \{\res{\rho_j}{X_t} < 0\}$ and $\Om_j' := \CU_j' \cap X_t = \{\res{\rho_j}{X_t} < -c_j\}$ form a pseudoconvex double cover of $X_t$ for all $t \in \BD_{1/2}$. Hence, these constants are fixed under such choice of a double covering; 
	\item $B'$: It can be obtained easily by restriction on each fibres;
	\item $V$: It follows from continuity of the total mass of the currents $(\om^n \w [X_t])_{t \in \overline{\BD}_{1/2}}$ (cf. \cite[Section 1.4]{Pan_2022}).
\end{itemize}
The remaining data is the main focus of Sections \ref{sec:unif_Skoda} and \ref{sec:vol_cap}:
\begin{itemize}
	\item $\af$, $A_\af$: These would be established in Proposition~\ref{prop:global_Skoda} assuming Conjecture \ref{conj:SL};
	\item $C_{\Kol,p}$: A uniform version of the volume-capacity comparison~\ref{bigasm:VC_comparison} will be treated by Proposition~\ref{prop:local_VC_comparison_in_family} and thus using Theorem~\ref{thm:Kolodziej} and uniform control of $\res{\rho_j}{X_t}$ for all $t \in \BD_{1/2}$, one can obtain a uniform constant $C_{\Kol,p}$.
\end{itemize}
These complete the proof of Theorem \ref{bigthm:uniform_a_priori_estimate}.

\section{Uniform Skoda's integrability theorem}\label{sec:unif_Skoda}
In this section, we follow the ideas in \cite{Zeriahi_2001, DGG2020} to establish a local version of uniform Skoda's integrability theorem in families.
Then we prove a uniform global version of Skoda's integrability theorem (i.e. geometric constants in \ref{bigasm:Skoda}) in a family which has a uniform $C_{SL} > 0$. 

\subsection{Local uniform Skoda's estimate}
Recall that $\pi: \CX \ra \BD$ is a proper surjective holomorphic map and satisfies the geometric setting \ref{geoset:irr_family}.
Then we fix some notations:
\begin{enumerate}
	\item $\CU$ is a strongly pseudoconvex domain 
	in $\CX$ and $\CU$ which is contained in a larger strongly pseudoconvex domain $\wCU \subset \CX$ such that $\pi(\wCU) \Subset \BD$; 
	\item $\rho$ (resp. $\wrho$) is a smooth strictly psh function defined in a neighborhood of $\overline{\CU}$ (resp. $\overline{\wCU}$) such that $\CU := \{\rho < 0\}$ (resp. $\wCU := \{\wrho < 0\}$) and $\CU_c = \{\rho < -c\}$ is relatively compact in $\CU$ for each $c > 0$;
	\item Fix a relatively compact subdomain $\CU' \Subset \CU$ such that $\CU' :=\{\rho < -c\}$ for some generic $c > 0$ with $\dd \rho$ non-vanishing on $\pl\CU'$;
	\item Define the slices by $\Om_t := \CU \cap X_t$, $\Om_t' := \CU' \cap X_t$, and $\wOm_t := \wCU \cap X_t$, which are strongly pseudoconvex domains in $X_t$.
	\item Fix $\om$ a hermitian metric which comes from a restriction of a hermitian metric in $\BC^N$. Let $\dd V_t = \res{\om^n}{X_t}$ be a smooth volume form defined on $\wOm_t$.
\end{enumerate}

Under such setup, we show the following Skoda-type estimate, independently of $t$.
\begin{thm}\label{thm:local_Skoda_in_family}
	Let $\CF_t$ be a family of negative psh functions defined on $\wOm_t$.
	Fix $r > 0$ such that $\BD_r \Subset \pi(\CU) \Subset \BD$.
	Assume that there is a uniform constant $C_F > 0$ such that for all $u_t \in \CF_t$,
	\[
	\int_{\Om_t} (-u_t) \dd V_t \leq C_F.
	\]
	Then there exist positive constants $\af$ and $A_\af$ which depend on $C_F$ such that for each $u_t \in \CF_t$,
	\[
	\int_{\Om_t'} \e^{-\af u_t} \dd V_t \leq A_\af.
	\]
	for all $t \in \BD_r$.
\end{thm}

\begin{rmk}
	We mostly use $\CU$ and $\CU'$ in the proof. 
	However, in order to approximate a negative psh function by a decreasing sequence of smooth psh functions, we need to shrink the domain in the very beginning of the proof.
	That is the reason why we choose $\CU$ lying in a bigger pseudoconvex domain $\wCU$. 
\end{rmk}

\begin{proof}
	We proceed in several steps: 
	
	\smallskip
	
	\noindent{\bf Step 0: Good covers.}
	First of all, we may assume that $u_t$ are smooth. 
	It follows indeed from a result of Fornaess--Narasimhan~\cite[Theorem 5.5]{FN_1980}, that one can approximate $u_t$ by a decreasing sequence of smooth non-positive psh functions on $\CU \cap X_t$. 
	We set 
	\[
	v_t := u_t + \rho_t
	\] 
	where $\rho_t = \res{\rho}{\Om_t}$.
	Since the masses of $\bt^n \w [X_t]$ are continuous in $t$, $\Vol(\Om_t,\dd V_t)$ is uniformly bounded by a constant $V' > 0$ up to shrinking $\BD$ (cf. \cite[Section 1.4]{Pan_2022}).
	By adding $\norm{\rho}_{L^\infty(\CU)} V'$ to $C_F$, we can also assume that 
	\[
	\forall u_t \in \CF_t, \quad
	\int_{\Om_t} (-v_t) \dd V_t \leq C_F.
	\]
	Choose a finite collection of balls of radius $2$, $B_{\BC^N}(p,2) \subset \BC^N$, centered at some point $p \in X_0$.
	We denote by $\BB_R := \CX \cap B_{\BC^N}(p,R) \Subset \CU$ for all $R \leq 2$.
	One may assume that the collection of balls of radius $1/2$, $\{\BB_{1/2}\}$, covers $\Om'_t$ for all $t \in \bBD_r$.
	For convenience, in this section, we fix constants $C_\rho, C_\om> 0$ satisfying $C_\rho^{-1} \ddc \rho \leq \ddc |z|^2 \leq C_\rho \ddc \rho$ and $C_\om^{-1} \om \leq \ddc |z|^2 \leq C_\om \om$ on each $\BB_2$

	\smallskip

	\noindent{\bf Step 1: 
		Poisson--Szeg\H{o} inequality.}
	{
		We first recall the following inequality
		\begin{align*}
			v_t(x) &\geq \int_\BB v_t (\ddc G_x)^n \w [X_t]
			= \int_{\BB \cap X_t} v_t (\ddc G_x)^n\\
			&= \underbrace{\int_{\BB \cap X_t} G_x (\ddc v_t) \w (\ddc G_x)^{n-1}}_{:= I_t(x)}
			+ \underbrace{\int_{\pl\BB \cap X_t} v_t \dc G_x \w (\ddc G_x)^{n-1}}_{:= J_t(x)},
		\end{align*}
		where $G_x(z) := \log |\Phi_x(z)|$ and $\Phi_x(z)$ is the automorphism of the unit ball $\BB$ that sends $x$ to the origin.
		The reader is referred to \cite[page 22-23]{DGG2020} for more details. 
	}

	\noindent{\bf Step 2: Control {$J_t$,} $I_t$ and Lelong numbers.}
	{
		Following the same proof in \cite[middle of page 23]{DGG2020}, we have $|J_t| \leq C_F C_1$ for some uniform $C_1 > 0$.
	}
	Now, we are going to treat the other more singular term 
	\[
	I_t(x) = \int_{\BB} G_x (\ddc v_t) \w (\ddc G_x)^{n-1} \w [X_t].
	\]
	In global K\"ahler setting, this part can be controlled by cohomology class of given K\"ahler metrics but it is not the case here.
	The spirit goes back to the local strategy in \cite{Zeriahi_2001} and Chern--Levine--Nirenberg inequality.
	
	\smallskip
	
	Consider the mass of the measure $\ddc v_t \w (\ddc G_x)^{n-1} \w [X_t]$
	\begin{align*}
		\gm_t(x) &:= \int_{\BB} \ddc v_t \w (\ddc G_x)^{n-1} \w [X_t]\\
		&= \underbrace{\int_{D(x,r)} \ddc v_t \w (\ddc G_x)^{n-1} \w [X_t]}_{:=I'_t(x)}
		+ \underbrace{\int_{\BB \setminus D(x,r)} \ddc v_t \w (\ddc G_x)^{n-1} \w [X_t]}_{:= I''_t(x)}
	\end{align*}
	where $D(x,r) := \set{\zt \in \BB}{|\Phi_x(\zt)| < r}$ for some $r \in (0,1)$.
	Note that $\mu_t = \frac{1}{\gm_t} \ddc v_t \w (\ddc G_x)^{n-1} \w [X_t]$ is a probability measure.
	
	\smallskip
	
	From the definition of $G_x(z)$, direct computation provides that 
	\begin{equation}\label{eq:lap green bdd}
		\ddc G_x(z) \leq C_2 \frac{\ddc|z|^2}{|\Phi_x(z)|^{2}}
	\end{equation}
	for some uniform constant $C_2 > 0$.
	Choose a cutoff function $\chi$ supported in $\BB_2$ and satisfying $\chi \equiv 1$ on $\BB$ and $-C_3 \ddc |z|^2 \leq \ddc \chi \leq C_3 \ddc |z|^2$ for some constant $C_3 > 0$.
	Then we have 
	\begin{align*}
		I''_t(x) &\leq \frac{1}{r^{2n-2}} \int_{\BB} \ddc v_t \w (\ddc |z|^2)^{n-1} \w [X_t]\\
		&\leq \frac{1}{r^{2n-2}} \int_{\BB_2} \chi \ddc v_t \w (\ddc |z|^2)^{n-1} \w [X_t]
		= \frac{1}{r^{2n-2}} \int_{\BB_2} v_t \ddc \chi \w (\ddc |z|^2)^{n-1} \w [X_t]\\
		&\leq \frac{C_3}{r^{2n-2}} \int_{\CU} (-v_t) (\ddc |z|^2)^{n} \w [X_t]
		= \frac{C_3}{r^{2n-2}} \int_{\Om_t} (-v_t) (\ddc |z|^2)^{n} 
		\leq \frac{C_3 C_F}{r^{2n-2}}.
	\end{align*}
	Because $|x| < 1/2$ in the setting, one may assume $D(x,r_0) \subset \BB_{3/4}$ for some uniform $r_0 > 0$ sufficiently small. 
	Hence, one get $I''_t(x) \leq \frac{C_3 C_F}{r_0^{2n-2}}$.
	
	\smallskip
	
	Consider cutoffs $(\chi_j)_{j=1}^n$ which are compactly supported on $\BB$ and satisfy:
	$\chi_1 \equiv 1$ on $\BB_{3/4}$, $\supp(\chi_1) \Subset \BB$, $\chi_{j+1} \equiv 1$ on $\supp(\chi_i)$ for every $j \in \{1,...,n-1\}$ and $-C_4 \ddc |z|^2 \leq \ddc\chi_j \leq C_4 \ddc |z|^2$ for some constant $C_4 > 0$.
	Using the trick in Chern--Levine--Nirenberg inequality, one can see that
	{\small
		\begin{align*}
			I'_t(x) &\leq \int_{\BB_{3/4}} \ddc v_t \w (\ddc G_x)^{n-1} \w [X_t]
			\leq \int_{\supp(\chi_1)} \chi_1 \ddc v_t \w (\ddc G_x)^{n-1} \w [X_t] \\
			&\leq \int_{\supp(\chi_1)} G_x \ddc v_t \w \ddc \chi_1 \w (\ddc G_x)^{n-2} \w [X_t]\\
			&\leq C_4 \lt( \sup_{z \in \BB \setminus \BB_{3/4}} |G_x(z)| \rt) \int_{\supp(\chi_1)}  \ddc v_t \w \ddc |z|^2 \w (\ddc G_x)^{n-2} \w [X_t]\\
			&\leq C_4 \lt( \sup_{z \in \BB \setminus \BB_{3/4}} |G_x(z)| \rt) \int_{\supp(\chi_2)}  \chi_2 \ddc v_t \w \ddc |z|^2 \w (\ddc G_x)^{n-2} \w [X_t]\\
			&\leq \cdots\\
			&\leq C_4^{n-1} \lt( \sup_{z \in \BB \setminus \BB_{3/4}} |G_x(z)| \rt)^{n-1} \int_{\supp(\chi_{n-1})} \ddc v_t \w (\ddc |z|^2)^{n-1} \w [X_t]\\
			&\leq C_4^{n} \lt( \sup_{z \in \BB \setminus \BB_{3/4}} |G_x(z)| \rt)^{n-1} \int_{\BB} (-v_t) \w (\ddc |z|^2)^{n} \w [X_t]\\
			&\leq C_4^{n} \lt( \sup_{z \in \BB \setminus \BB_{3/4}} |G_x(z)| \rt)^{n-1} \int_{\Om_t} (-v_t) \w (\ddc |z|^2)^{n}
			\leq C_4^n C_F \lt( \sup_{z \in \BB \setminus \BB_{3/4}} |G_x(z)| \rt)^{n-1}.
	\end{align*}}
	Since $|x| < 1/2$ and $\Phi_x(z)$ moves smoothly in $x$ and $z$, there is a uniform constant $C_5$ such that 
	\[
	\lt(\sup_{z \in \BB \setminus \BB_{3/4}} |G_x(z)|\rt) \leq C_5.
	\]
	Thus, $I'_t$ is uniformly bounded from above by the constant $C_4^n C_F C_5^{n-1}$. 
	
	\smallskip
	
	Estimates of $I'$ and $I''$ yield a constant $\nu = \frac{C_F C_3}{r_0^{2n-2}} + C_F C_4^n C_5^{n-1}$ such that $\gm_t(x)$ is bounded by $\nu$ from above.
	{
		On the other hand, we have the lower bound of $\gm_t \geq C^{-1}_\rho$ by similar computation in \cite[bottom of page 23]{DGG2020}.
	}
	Therefore, we obtain a two-sided bound of $\gm_t$:
	\begin{equation}\label{eq:Lelong bound}
		\forall x \in \BB_{1/2} \,\,\text{and}\,\, \forall t \in \overline{\BD}_{r}, \quad
		C_\rho^{-1} \leq \gm_t(x) \leq \nu.
	\end{equation}

	\noindent{\bf Step 3: Conclusion.}
	{
		We closely follow the strategy in \cite[page 24-25]{DGG2020} to conclude.
	}
	Combining (\ref{eq:lap green bdd}), (\ref{eq:Lelong bound}) and Jensen's inequality, we derive
	\begin{align*}
		\e^{-\af I_t(x)}
		&= \exp\lt( \int_{z \in \BB} -\af \gm_t(x) G_x \dd\mu_t \rt)
		\leq C_\rho \int_{z \in \BB} \frac{\ddc v_t \w (\ddc |z|^2)^{n-1} \w [X_t]}{|\Phi_x(z)|^{\af \nu + 2n -2}} .
	\end{align*}
	Integrating $x \in \BB_{1/2}$ and using Fubini's theorem, one can infer the following inequality
	\begin{align*}
		&\int_{x \in \BB_{1/2}} \e^{-\af u_t} (\ddc |x|^2)^n \w [X_t]
		\leq \int_{x \in \BB_{1/2}} \e^{-\af v_t} (\ddc |x|^2)^n \w [X_t]\\
		&\leq \e^{\af C_FC_1} C_\rho \int_{z \in \BB} 
		\lt(\int_{x \in \BB_{1/2}} 
		\frac{(\ddc |x|^2)^n \w [X_t]}{|\Phi_x(z)|^{\af \nu + 2n -2}}\rt)
		\ddc v_t(z) \w (\ddc |z|^2)^{n-1} \w [X_t]
	\end{align*}
	Fix a constant $\af$ sufficiently small such that $\af \nu < 2$ and put $\bt = \frac{2-\af\nu}{2n} > 0$.
	Following the similar proof of \cite[Lemma 2.13]{DGG2020},
	we have
	\[
	C_\bt^{-1} (\ddc_x |\Phi_x(z)|^{2\bt})^n
	\leq \frac{(\ddc|x|^2)^n}{|\Phi_x(z)|^{\af\nu + 2n -2}} 
	\leq C_\bt (\ddc_x |\Phi_x(z)|^{2\bt})^n.
	\]
	Now, using the same trick in the proof of Chern--Levine--Nirenberg inequality, one has
	\begin{align*}
		\int_{x \in \BB_{1/2}} \frac{(\ddc|x|^2)^n \w [X_t]}{|\Phi_x(z)|^{\af\nu + 2n -2}}
		&\leq C_\bt \int_{x \in \BB_{1/2}} (\ddc_x |\Phi_x(z)|^{2\bt})^n \w [X_t]\\
		&\leq C_\bt \int_{x \in \supp(\chi_1)} \chi_1 (\ddc_x |\Phi_x(z)|^{2\bt})^n \w [X_t]\\
		&\leq C_4^n C_\bt \int_{x \in \BB} (\ddc |x|^2)^n \w [X_t]
		\leq C_4^n C_\om^n C_\bt \Vol_{\om_t}(\Om_t).
	\end{align*}
	{
		Note that in global K\"ahler cases, the above estimate is controlled by cohomology classes.}
	Finally, we get the estimate on each $\BB_{1/2}$
	{\small
		\begin{align*}
			\int_{x \in \BB_{1/2}} \e^{-\af u_t} (\ddc |x|^2)^n \w [X_t]
			&\leq \e^{\af C_F C_1} \int_{z \in \BB} \lt(C_4^n C_\om^n C_\bt \Vol_{\om_t}(\Om_t)\rt) \ddc v_t \w (\ddc |z|^2)^{n-1} \w [X_t]\\
			&\leq \e^{\af C_F C_1} C_4^n C_\om^n C_\bt \Vol_{\om_t}(\Om_t) 
			\int_{\BB_2} \chi \ddc v_t \w (\ddc |z|^2) \w [X_t]\\
			&\leq \e^{\af C_F C_1} C_4^{n+1} C_\om^n C_\bt \Vol_{\om_t}(\Om_t) 
			\int_{\Om_t} (-v_t) (\ddc|z|^2)^n\\
			&\leq \e^{\af C_F C_1} C_F C_4^{n+1} C_\om^n C_\bt \Vol_{\om_t}(\Om_t).
	\end{align*}}
	{
		Note that $t \mapsto \Vol_{\om_t} (\Om_t)$ is continuous.
		One has a uniform control $\int_{x \in \BB_{1/2}} \e^{-\af u_t} (\ddc |x|^2)^n \w [X_t]$ for all $t$ close to $0$.
	}
	Summing the integration on every $\BB_{1/2}$ in the collection, we obtain the estimate in Theorem~\ref{thm:local_Skoda_in_family} as desired.
\end{proof}

\subsection{Global Skoda's estimate}
Now, we assume that there is a uniform constant $C_{SL} > 0$ such that $X_t$ satisfies \ref{bigasm:SL} for all $t \in \BD$.
As a consequence, we have the following uniform global version of Skoda's estimate:

\begin{prop}\label{prop:global_Skoda}
	Assume that there is a uniform constant $C_{SL} > 0$ such that for all $t \in \BD$ and for every $u_t \in \PSH(X_t,\om_t)$ with $\sup_{X_t} = 0$,
	\[
	\frac{1}{V_t} \int_{X_t} (-u_t) \om_t^n \leq C_{SL}
	\]
	where $V_t := \Vol_{\om_t} (X_t)$.
	Then there exists constants $\af$, $A_\af$ such that for all $t \in \overline{\BD}_{1/2}$ and for all $u_t \in \PSH(X_t, \om_t)$ with $\sup_{X_t} u_t = 0$,
	\[
	\int_{X_t} \e^{-\af u_t} \om_t^n \leq A_\af.
	\]
\end{prop}

\begin{proof}
	Without loss of generality, we just treat the proof for $t$ in a small neighborhood near $0 \in \BD$.
	Let $(\CU_j)_{j \in J}$ and $(\CU_j')_{j \in J}$ be a strongly pseudoconvex finite double cover of $\pi^{-1}(\BD_r)$ for some $r>0$ sufficiently small. 
	We write $\CU_j' := \{\rho_j < -c_j\} \Subset \CU_j = \{\rho_j < 0\}$ for some $c_j > 0$.
	For simplicity, we may assume that $\ddc \rho_j \geq \om$ for all $j \in J$.
	Also, we set the slices $\Om_{t,j} := X_t \cap \CU_j$ and $\Om_{t,j}' := X_t \cap \CU_j'$.
	For all $u_t \in \PSH(X_t, \om_t)$ with $\sup_{X_t} u_t = 0$, it is obvious that $u_t + \rho_{t,j} \in \PSH(\Om_{t,j})$.
	Note that for all $j \in J$,
	\[
	\int_{\Om_{t,j}} -(u_t + \rho_{t,j}) \om_t^n 
	\leq \int_{X_t} -u_t \om_t^n + V_t \norm{\rho_j}_{L^\infty(\CU_j)}
	\leq C_F
	\]
	for some uniform constant $C_F$.
	By Theorem~\ref{thm:local_Skoda_in_family}, we obtain
	\[
	\forall j \in J, \quad
	\int_{\Om_{t,j}'} \e^{-\af_j u_t} \om_t^n 
	\leq \int_{\Om_{t,j}'} \e^{-\af_j (u_t + \rho_{t,j})} \om_t^n 
	\leq A_{\af_j, j}.
	\]
	Since $J$ is a finite set of indices, one can easily derive the desired estimate.
\end{proof}
This ensures that one can find uniform geometric constants~\ref{bigasm:Skoda} in a family which fulfills Conjecture~\ref{conj:SL}.

\section{Uniform volume-capacity comparison}\label{sec:vol_cap}
This section aims to deal with the volume-capacity comparison~\ref{bigasm:VC_comparison} in a family $\pi: \CX \ra \BD$ with locally irreducible fibres.
Let $\CU = \{\rho < 0\}$ be a strongly pseudoconvex domain in $\CX$. 
We also assume that the closure of $\CU$ can be contained in a larger strongly pseudoconvex set $\wCU = \{\wrho < 0\}$ in $\CX$. 
Denote the slices by $\Om_t := \CU \cap X_t$ and $\wOm_t := \wCU \cap X_t$.

\subsection{Subextensions and relative extremal functions}
We first recall some useful facts in pluripotential theory:

\smallskip

\subsubsection{Subextensions}
Fix a strongly pseudoconvex domain $\Om = \{\rho < 0\}$ and a strongly pseudoconvex neighborhood $\wOm=\{\wrho < 0\}$ containing $\overline{\Om}$.
We define
\[
\CE^0(\Om) 
:= \set{u \in \PSH(\Om) \cap L^\infty(\Om)}{ 
	\res{u}{\pl\Om} = 0 
	\text{ and } 
	\int_{\Om} (\ddc u)^n < +\infty}.
\]
We recall some properties of subextension of the plurisubharmonic functions in $\CE^0(\Om)$.
\begin{lem}[{\cite[Theorem 2.2]{Cegrell_Zeriahi_2003} and \cite[Lemma 1.7]{Guedj_Guenancia_Zeriahi_2020}}]\label{lem:subextension}
	Suppose that $\vph \in \CE^0(\Om)$. 
	The subextension of $\vph$ is defined as: 
	\[
	\wvph := \sup \set{u \in \PSH(\wOm)}{\res{u}{\pl\wOm} \leq 0 \text{ and } u \leq \vph \text{ in } \Om}. 
	\]
	Then the subextension $\wvph$ satisfies the following properties: 
	\begin{enumerate}
		\item $\wvph \in \CE^0(\wOm)$;
		\item $\wvph \leq \vph$ on $\Om$;
		\item $\int_{\wOm} (\ddc \wvph)^n \leq \int_\Om (\ddc \vph)^n$.
	\end{enumerate}
\end{lem}

\smallskip

\subsubsection{Relative extremal functions}
Next, we review the definition and some basic properties of relative extremal functions (cf. \cite[Chapter 3]{GZbook}):

\begin{defnprop}\label{defnprop:relative_extremal_function}
	Let $E$ be a Borel subset in a strongly pseudoconvex domain $\Om$.
	The relative extremal function with respect to $(E,\Om)$ is defined as follows
	\[
	h_{E;\Om}(z) = \sup \set{u(z)}{u \in \PSH(\Om),\, u \leq 0 \text{ and } u|_E \leq -1}.
	\]
	Suppose that $E$ is a relatively compact Borel subset in $\Om$.
	Then we have the following facts:
	\begin{enumerate}
		\item The function $h^\ast_{E;\Om}$ is psh in $\Om$ and $h^\ast_{E;\Om}= -1$ on $E$ off a pluripolar subset;
		\item The boundary value of $h^\ast_{E;\Om}$ is zero (i.e. $\lim_{z \ra \pl\Om} h^\ast_{E;\Om}(z) = 0$);
		\item The Monge--Amp\`ere measure $(\ddc h^\ast_{E;\Om})^n$ puts no mass on $\Om \setminus \ov{E}$;
		\item The capacity can be expressed by the integration of Monge--Amp\`ere measure of the relative extremal function: $\CAP(E;\Om) = \int_\Om (\ddc h^\ast_{E;\Om})^n$.
	\end{enumerate}
\end{defnprop}

\smallskip

\subsection{Volume-capacity comparison}
Then we explain the volume-capacity comparison in families:

\begin{prop}\label{prop:local_VC_comparison_in_family}
	For every $k>1$, there exists a constant $C_{VC,k}$ such that 
	\[
	\forall K_t \Subset \Om_t, \quad
	\Vol(K_t) \leq C_{VC,k} \CAP^k (K_t;\Om_t),
	\]
	for all $t \in \BD_{1/2}$
\end{prop}

In smooth setting, the proof of the volume-capacity comparison was first given by Ko{\l}odziej~\cite{Kolodziej_2002}.
For global semi-positive setup, Guedj and Zeriahi \cite{Guedj_Zeriahi_2005} provided a proof simply using Skoda's integrability theorem and the comparison of Bedfor--Taylor and Alexander--Taylor capacities. 
For local cases, in \cite[Lemma 1.9]{Guedj_Guenancia_Zeriahi_2020}, Guedj--Guenancia--Zeriahi gave an easy proof for us to chase the depending constants. 
We also compute explicitly the constant $C_{VC,k}$ to verify that it is independent of $t$.

\begin{proof}
	Without loss of generality, we may assume that $K_t$ is non-pluripolar.
	Otherwise, both sides of the inequality are zeros.
	We define
	\[
	u_{K_t} = \frac{h^\ast_{K_t;\Om_t}}{\CAP^{1/n}(K_t;\Om_t)}.
	\]
	According to Proposition~\ref{defnprop:relative_extremal_function}, one has $u_{K_t} \in \CE^0(\Om_t)$ and $\int_{\Om_t} (\ddc u_{K_t})^n = 1$.
	Recall that $\wCU = \{\wrho < 0\}$ is a strongly pseudoconvex neighborhood of $\overline{\CU}$ and $\wOm_t = \wCU \cap X_t$.
	Let $C_{\wrho} > 0$ be a constant such that $\dd V \leq C_{\wrho} (\ddc \wrho)^n$ on $\wCU$. 
	Consider the subextension of $u_{K_t}$:
	\[
	\wu_{K_t} = \sup\set{u \in \PSH(\wOm_t) \cap L^\infty(\wOm_t)}{\res{u}{\pl\wOm_t} \leq 0 \text{ and } u \leq u_{K_t} \text{ in }\Om_t}.
	\]
	By Lemma~\ref{lem:subextension}, we have $\wu_{K_t} \in \CE^0(\wOm_t)$, $\wu_{K_t} \leq u_{K_t}$ in $\Om_t$, and $\int_{\wOm_t} (\ddc \wu_{K_t})^n \leq \int_{\Om_t} (\ddc u_{K_t})^n = 1$.
	Using the integration by parts and the condition $\int_{\wOm_t} (\ddc \wu_{K_t})^n \leq 1$, one can see that $\norm{\wu_{K_t}}_{L^1(\wOm_t)}$ is uniformly bounded independent of $K_t$.
	Indeed, 
	\begin{align*}
		\int_{\wOm_t} (-\wu_{K_t}) \dd V
		&\leq \Vol^{\frac{n-1}{n}}(\wOm_t) \lt( \int_{\wOm_t} (-\wu_{K_t})^n \dd V \rt)^{1/n}\\
		&\leq \Vol^{\frac{n-1}{n}}(\wOm_t) C_\wrho \lt( \int_{\wOm_t} (-\wu_{K_t})^n (\ddc \wrho_t)^n \rt)^{1/n}\\
		&\leq \Vol^{\frac{n-1}{n}}(\wOm_t) C_\wrho \norm{\wrho_t}_{L^\infty(\wOm_t)} 
		\lt( \int_{\wOm_t} (\ddc \wu_{K_t})^n \rt)^{1/n}
		\leq \Vol^{\frac{n-1}{n}}(\wOm_t) C_\wrho \norm{\wrho}_{L^\infty(\wCU)}.
	\end{align*}
	According to Theorem~\ref{thm:local_Skoda_in_family}, there exists constants $\af, A_\af > 0$ such that for all $K_t \Subset \Om_t$ non-pluripolar,
	\[
	\int_{\Om_t} \e^{-\af \wu_{K_t}} \dd V \leq A_\af.
	\]
	Recall that $h_{K_t}^\ast = -1$ on $K_t$ almost everywhere.
	By the definition of $u_{K_t}$ and $\wu_{K_t} \leq u_{K_t}$, we have
	\[
	\Vol(K_t) \cdot \exp\lt( \frac{\af}{\CAP^{1/n}(K_t;\Om_t)} \rt)
	= \int_{K_t} \e^{- \af u_{K_t}} \dd V
	\leq \int_{\Om_t} \e^{- \af u_{K_t}} \dd V
	\leq A_\af
	\]
	and this implies
	\[
	\Vol(K_t) \leq A_\af \exp\lt( -\frac{\af}{\CAP^{1/n}(K_t;\Om_t)} \rt)
	\leq A_\af \frac{b_k}{\af^{kn}} \CAP^k(K_t;\Om_t).
	\]
	where $b_k$ is a numerical constant such that $\exp(-1/x) \leq b_k x^{kn}$ for all $x > 0$.
\end{proof}

\subsection{Global volume-capacity comparison}
In this section, we show the uniform volume-capacity comparison in a given family of compact hermitian varieties $\pi: (\CX, \om) \ra \BD$ with locally irreducible fibres. 

Let $(X,\om)$ be a compact hermitian variety.
We define the similar concept of Monge--Amp\`ere capacity with respect to $\om$-psh functions by
\[
\CAP_\om(K) := \sup\set{\int_K (\om + \ddc u)^n}{u \in \PSH(X,\om), \text{ and } 0 \leq u \leq 1}.
\]
On the other hand, fixing a pseudoconvex finite double cover $(\Om_j)_j$ and $(\Om'_j)_j$ of $X$ such that $\Om'_j \Subset \Om_j$, we define the Bedford--Taylor capacity by 
\[
\CAPBT(K) := \sum_{j} \CAP(K \cap \overline{\Om_j'}; \Om_j).
\]
To prove the global volume-capacity comparison, we should first compare the Bedford--Taylor capacity and the capacity of $\om$-psh functions.

\begin{lem}\label{lem:loc_global_cap_comparison}
	There exists a constant $C_{BT,\om} > 0$ such that 
	\[
	\forall \text{ compact subset }K_t \subset X_t, \quad 
	C_{BT,\om}^{-1} \CAPBT(K_t) \leq \CAP_{\om_t} (K_t) \leq C_{BT,\om} \CAPBT(K_t)
	\]
	for all $t \in \overline{\BD}_{1/2}$.
\end{lem}

\begin{proof}
	On a fixed compact K\"{a}hler manifold, a similar version of Lemma~\ref{lem:loc_global_cap_comparison} was provided by Ko{\l}odziej \cite[Section~1]{Kolodziej_2003}.
	The proof of Lemma~\ref{lem:loc_global_cap_comparison} is similar to Ko{\l}odziej's proof. 
	For the reader's convenience, we include the proof here.
	
	After shrinking $\BD$, we may assume that $(\CU_j)_j$ and $(\CU_j')_j$ form a pseudoconvex double cover of $\CX$ such that $\CU_j = \{\rho_j < 0\}$ for some strictly psh function $\rho_j$ and $\CU_j = \{\rho_j < -c_j\}$. 
	Multiplying a positive constant, we may assume that $\frac{1}{C} \ddc_\CX \rho_j \leq \om \leq C \ddc_\CX \rho_j$.
	Let $C' > 0$ be a constant so that $0 \leq \rho_j \leq C'$ for all $j$.
	
	Write $\Om_{j,t} = \CU_j \cap X_t$ and $\Om_{j,t}' = \CU_j' \cap X_t$.
	Fix $u_t \in \PSH(X_t, \om_t)$ with $0 \leq u_t \leq 1$.
	Let $K_t$ be a compact subset of $X_t$ and $K_{t,j} := K_t \cap \overline{\Om_j'}$.
	Then we have
	\begin{align*}
		\int_{K_t} (\om_t + \ddc_t u_t)^n
		\leq& \sum_j \int_{K_{j,t}} (\ddc_t (C \rho_j + u_t))^n\\
		=& \sum_j \int_{K_{j,t}} (CC'+1)^n \lt(\ddc_t \lt(\frac{C \rho_j + u_t}{CC'+1}\rt) \rt)^n\\
		\leq& (CC'+1)^n \sum_j \CAP(K_{j,t}; \Om_j)
		= (CC'+1)^n \CAPBT(K_t)
	\end{align*}
	and hence
	$\CAP_{\om_t}(K_t) \leq (CC'+1)^n \CAPBT(K_t)$.
	
	On the other hand, we shall use the gluing argument to prove the local capacity is bounded by the global capacity.
	Suppose $u_t \in \PSH(\Om_{j,t})$ and $0 \leq u_t \leq 1$.
	Consider a smooth function $\chi_j$ defined on $\CU_j$ such that
	\[
	\chi(z) = 
	\begin{cases}
		-1 & \text{when } z \in \CU'_j\\
		2 & \text{when $z$ is in a neighborhood of $\pl \CU_j$}
	\end{cases}.
	\]
	We can find a small $\dt_j \in (0, \frac{1}{3})$ such that $\dt_j \chi_j$ can be extended to a $\om$-psh function on $\CX$.
	Now, we us the gluing argument to define a $\om_t$-psh function $\psi$ and it is identically equal to $\dt_j u_t$ in $\Om'_{j,t}$
	\[
	\psi(z) = \begin{cases}
		\dt_j u(z) & \text{when }z \in \Om'_{j,t}\\
		\max\{\dt_j\res{\chi_j}{X_t}, \dt_j u_t\} & \text{when } z \in \Om_{j,t} \setminus \Om'_{j,t}\\
		\dt_j\chi_j & \text{when } z \in X_t \setminus \Om_{j,t}\\
	\end{cases}.
	\]
	Obviously, $\tde{\psi} = \psi + 1/3$ is a $\om_t$-psh function and $0 \leq \tde{\psi} \leq 1$.
	Then, we obtain
	\[
	\int_{K_{j,t}} (\ddc_t u_t)^n
	= \frac{1}{\dt_j^n} \int_{K_{j,t}} (\ddc_t \dt_j u_t)^n
	\leq \frac{1}{\dt_j^n} \int_{K_t} (\om_t + \ddc_t \tde{\psi})^n
	\leq \frac{1}{\dt_j^n} \CAP_{\om_t}(K_t).
	\]
\end{proof}

Combining Proposition~\ref{prop:local_VC_comparison_in_family} and Lemma~\ref{lem:loc_global_cap_comparison}, we have the global volume-capacity comparison in families:
\begin{prop}\label{prop:global_VC_comparison}
	Given $k > 1$, there exists a uniform constant $C_{GVC,k} \geq 1$ such that 
	\[
	\forall \text{ compact subset } K_t \subset X_t, \quad 
	\Vol_{\om_t} (K_t) \leq C_{GVC,k} \CAP^k_{\om_t}(K_t)
	\]
	for all $t \in \BD_{1/2}$.
\end{prop}

\section{Sup-$L^1$ comparison in families}\label{sec:L1_estimate}
In this section, we pay a special attention to Conjecture~\ref{conj:SL}.
We shall follow the strategy of proof in \cite[Section 3]{DGG2020}.
First of all, we recall the assumption which will be used in this section: 
\begin{asm}[=Geometric assumption~\ref{geoasm:L1_assumption}]
	Suppose that $\pi: \CX \ra \BD$ is a family of hermitian varieties satisfies the geometric setting \ref{geoset:irr_family} and one of the following conditions:
	\begin{enumerate}[label=(\alph*)]
		\item $\pi$ is locally trivial;
		\item $\pi: \CX \ra \BD$ is a smoothing of $X_0$ and $X_0$ has only isolated singularities.
	\end{enumerate} 
\end{asm}
Under such assumption, we establish a uniform $L^1$-estimate of $\om_t$-psh function: 
\begin{prop}[=Proposition \ref{bigprop:uniform_L1_isolated_singularities}]\label{prop:L^1_estimate}
	Suppose that $\pi: \CX \ra \BD$ satisfies the geometric assumption~\ref{geoasm:L1_assumption}.
	Then there exists a uniform constant $C>0$ such that for all $t \in \BD_{1/2}$ 
	\[
	\forall \vph_t \in \PSH(X_t, \om_t), \quad
	\sup_{X_t} \vph_t - C \leq \frac{1}{V_t} \int_{X_t} \vph_t \om_t^n \leq \sup_{X_t} \vph_t,
	\]
	where $V_t := \Vol_{\om_t}(X_t)$.
\end{prop}

{
	\subsection{Proof of Proposition~\ref{bigprop:uniform_L1_isolated_singularities}}
	
	\subsubsection{Locally trivial families}
	In locally trivial cases, one can follow the similar strategy in \cite[Section 3.2]{DGG2020} for the proof. 
	The only difference is to replace the local potentials of $\om_0$ by the local inequalities $0 < C_\rho^{-1} \ddc\rho \leq \om_0 \leq C_\rho \ddc\rho$ for some local psh function $\rho$ and some constant $C_\rho > 0$ on $X_0$. 
	
	\subsubsection{Smoothing of varieties with isolated singularities}
	Before diving into the main goal, we recall the uniform boundedness of the integral of Laplacian.
	\begin{lem}\label{lem:Laplacian_estimate_smoothing}
		Suppose that $\pi: (\CX,\om) \ra \BD$ satisfies the geometric setting \ref{geoset:irr_family} and it is a smoothing of a compact variety $X_0$.
		There is a uniform constant $C_{\Lap} > 0$ such that
		\[
		\forall \vph_t \in \PSH(X_t, \om_t), \quad
		\int_{X_t} (\om_t + \ddc_t \vph_t) \w \om_t^{n-1} \leq C_{\Lap},
		\]
		for all $t \in \BD_{1/2}$.
	\end{lem}
	
	\begin{proof}
		Recall that from \cite[Theorem A]{Pan_2022}, there is a uniform constant $C_G > 0$ such that for all $t \in \BD_{1/2}^\ast$ the normalized Gauduchon factor $g_t$ with respect to $(X_t, \om_t)$ (i.e. $\ddc_t \lt( g_t \om_t^{n-1} \rt) = 0$) is bounded between $1$ and $C_G$.
		Then we have 
		\begin{align*}
			\int_{X_t} (\om_t + \ddc_t \vph_t) \w \om_t^{n-1} 
			&\leq \int_{X_t} (\om_t + \ddc_t \vph_t) \w g_t \om_t^{n-1}\\
			&\leq \int_{X_t} g_t \om_t^n + \int_{X_t} \vph_t \ddc_t(g_t \om_t^{n-1})
			\leq C_G \Vol_{\om_t}(X_t)
		\end{align*}
		for all $t \in \BD_{1/2}^{\ast}$.
		Since $(\Vol_{\om_t}(X_t))_{t \in \BD_{1/2}}$ is uniformly bounded from above (cf. \cite[Section 1.4]{Pan_2022}), we have a desired estimate for all $t \in \BD_{1/2}^\ast$.
		
		\smallskip
		
		On the central fibre $X_0$, there always exists a constant $C_{SL,0}>0$ such that 
		\[
		\frac{1}{\Vol_{\om_0}(X_0)} \int_{X_0} -\vph_0 \om_0^n \leq C_{SL,0}
		\] 
		for all $\vph_0 \in \PSH(X_0, \om_0)$ with $\sup_{X_0} \vph_0 = 0$.
		For all $\vph_0 \in \PSH(X_0, \om_0)$, we take $\tilde{\vph}_0 = \vph_0 - \sup_{X_0} \vph_0$. 
		Then one can use the argument as in Lemma~\ref{lem:upper_bound_c} to find a constant $C_{\Lap}' > 0$ such that $\int_{X_0} (\om_0 + \ddc_0 \vph_0) \w \om_0^{n-1} \leq C_{\Lap}'$.
		We have thus obtained the desired estimate. 
	\end{proof}
	
	We follow the idea in \cite[Section 3.3]{DGG2020} to get the proof.
	Note that we only need to take care of sequences of sup-normalized $\om_{t_k}$-psh functions $(\vph_{t_k})_k$ where $t_k \xrightarrow[k \ra +\infty]{} 0$.
	\smallskip
	
	\noindent{\bf Step 1: Choose a good covering and a test function.}
	Following the same argument in \cite[page 30, Step 2]{DGG2020}, up to shrinking $\BD$, we can find a finite open covering $(\CV_i)_{i \in I}$ of $\CX$ such that 
	\begin{enumerate}
		\item each point of $\CZ := \CX^\sing = X_0^\sing$ belongs to exactly one element of $\CV_i$ of the covering, we denote by $J$ the collection of indices of these open subsets;
		\item on each $\CV_i$, we have a smooth strictly psh function $\rho_i$ such that
		\[
		C_\rho^{-1} \ddc \rho_i \leq \om \leq C_\rho \ddc \rho_i
		\quad\text{and}\quad
		0 \leq \rho_i \leq C_\rho
		\]
		for a uniform constant $C_\rho > 0$;
		\item for each $i \in J$, there is a relatively compact open subset $\CW_i \Subset \CV_i$ with $\CW_i \cap \CZ \neq \emptyset$.
	\end{enumerate}
	Define 
	\[
	\dt := \oh \min_{i \in J} \lt\{ \dist_\om (\pl\CW_i, \CW_i \cap X_0^\sing) \rt\} > 0.
	\]
	Let $\chi_i$ be a cut-off function supported in $\CV_i$ and $\chi_i \equiv 1$ in a neighborhood of $\CW_i$.
	Set $\rho = \sum_{i \in I} \chi_i \rho_i$.
	Obviously, we have $\om \leq C_\rho \ddc \rho$ on $\CW := \cup_{i \in I} \CW_i$.
	Furthermore, we may assume $-C_\rho \om \leq \ddc \rho \leq C_\rho \om$ on $\CX$ by choosing larger $C_\rho$.
	
	\smallskip
	
	\noindent{\bf Step 2: Uniform $L^1$-estimate away from singularities.}
	Define a set
	\[
	\CR := \set{p \in \CX}{\dist_\om(p,X_0^\sing) > \dt/2}.
	\]
	Since $\CR^c$ lies in $\CW$, after shrinking $\BD$, we can cover $\CR$ by finitely many open subsets $(\CU_i)_i$ away from the singular locus. 
	We may assume that $\pi$ is locally trivial on $\CR$ with respect to $(\CU_i)_i$ because $\pi$ is a submersion on $\CR$.
	
	\smallskip
	Following the argument in \cite[page 31, Step 3]{DGG2020}, one can prove that 
	there is a constant $C>0$ and a subsequence of $(t_k)_k$ such that 
	\[
	\sup_{\CR \cap X_{t_k}} \vph_{t_k} \geq -C.
	\]
	By the irreducibility of $X_0$, $\CR$ is connected.
	Then one can use the same proof in locally trivial cases to show that 
	\[
	\int_{\CR \cap X_{t_k}} (-\vph_{t_k}) \om^n_{t_k} \leq C_\CR
	\]
	for some uniform constant $C_\CR > 0$.
	
	\smallskip
	
	\noindent{\bf Step 3: Conclusion.}
	Recall that on $\CW$ we have $\om \leq C_\rho \ddc \rho$.
	Define a smooth $(n,n)$-form $\Om := \om^n - C_\rho^n (\ddc \rho)^n$.
	It is easy to see that $\res{\Om}{\CW \cap X_t} \leq 0$ and $\Om_t := \res{\Om}{X_t} \leq C_\Om \om_t^n$ for some uniform constant $C_\Om > 0$.
	Note that $\CR^c \subset \CW$.
	We have
	\begin{align*}
		\int_{X_{t_k}} (-\vph_{t_k}) \Om_{t_k} 
		&= \int_{\CR \cap X_{t_k}} (-\vph_{t_k}) \Om_{t_k}
		+ \int_{\CR^c \cap X_{t_k}} (-\vph_{t_k}) \Om_{t_k}\\
		&\leq C_\Om \int_{\CR \cap X_{t_k}} (-\vph_{t_k}) \om_{t_k}^n
		\leq C_\CR C_\Om.
	\end{align*}
	On the other hand, we have 
	\begin{align*}
		\int_{X_{t_k}} (-\vph_{t_k}) (\ddc \rho)^n
		&= \int_{X_{t_k}} - \rho \ddc \vph_{t_k} \w (\ddc \rho)^{n-1}\\
		&= - \int_{X_{t_k}} \rho (\om_{t_k} + \ddc \vph_{t_k}) \w (\ddc \rho)^{n-1}
		+ \int_{X_{t_k}} \rho \om_{t_k} \w (\ddc \rho)^{n-1}\\
		&\leq C_\rho^n \int_{X_{t_k}} (\om_{t_k} + \ddc \vph_{t_k}) \w \om_{t_k}^{n-1} + C_\rho^n \Vol_{\om_{t_k}}(X_{t_k})\\
		&\leq C_\rho^n C_{\Lap} + C_\rho^n \Vol_{\om_{t_k}}(X_{t_k}).
	\end{align*}
	The fourth line comes directly from Lemma~\ref{lem:Laplacian_estimate_smoothing}.
	All in all, we obtain a uniform $L^1$-estimate of sup-normalized $\om_{t_k}$-psh functions: 
	\begin{align*}
		\int_{X_{t_k}} (-\vph_{t_k}) \om_{t_k}^n
		&= C_\rho^n \int_{X_{t_k}} (-\vph_{t_k}) (\ddc \rho)^n 
		+ \int_{X_{t_k}} (-\vph_{t_k}) \Om_{t_k}
		\leq C_\rho^{2n} C_{\Lap} + C_\rho^{2n} \Vol_{\om_{t_k}}(X_{t_k}) + C_\CR C_\Om.
	\end{align*}
}

\section{Families of Calabi--Yau varieties}\label{sec:geometric_applications}
In this final section, we apply our results to families of Calabi--Yau varieties.
By Calabi--Yau variety, we mean a normal variety with canonical singularities and trivial canonical bundle.

\subsection{Canonical singularities}
We first recall the notion of canonical singularities. 
The reader is referred to \cite[Section 5]{EGZ_2009} for more details.
Let $X$ be a normal variety.
We say that $X$ has \textit{canonical singularities} if the pluricanonical sheaf $K_X^{[N]}$ is locally free for some $N \in \BN$, and for any resolution of singularities $p: \wX \ra X$, for any local generator $\af$ of $K_X^{[N]}$, the meromorphic pluricanonical form $p^\ast \af$ is holomorphic.

\begin{lem}\label{lem:AG_and_fibration}
	Suppose that $\pi: \CX \ra \BD$ is a proper surjective holomorphic map satisfying the geometric setting~\ref{geoset:irr_family}. 
	In addition, assume that 
	\begin{itemize}
		\item $\CX$ is normal and $K_\CX$ (or equivalently $K_{\CX/\BD}$) is locally free;
		\item for each $t \in \BD$, $X_t$ has only canonical singularities.
	\end{itemize}
	Then the following are equivalent
	\begin{itemize}
		\item $K_{\CX/\BD}$ (or $K_\CX$) is trivial up to shrinking $\BD$; 
		\item $K_{X_t}$ is trivial for all $t$ small (i.e. $X_t$ is Calabi--Yau for all $t$ small).
	\end{itemize}
\end{lem}

\begin{proof}
	We include some arguments here for the reader's convenience.
	Suppose that $K_{\CX/\BD}$ (or $K_\CX$) is trivial. 
	Then we have 
	\[
	K_{X_t^\reg} \simeq \res{K_{\CX^\reg/\BD}}{X_t^\reg} 
	= \res{(K_{\CX^\reg} - \pi^\ast K_{\BD})}{X_t^\reg} 
	\simeq \res{\CO_{\CX^\reg}}{X_t^\reg} \simeq \CO_{X_t^\reg}.
	\]
	When $t$ is close to $0$, since $X_t$ is normal and $K_{X_t}$ is reflexive, we have 
	\[
	K_{X_t} \simeq (j_t)_\ast(K_{X_t^\reg}) \simeq (j_t)_\ast (\CO_{X_t^\reg}) \simeq \CO_{X_t}
	\]
	where $j_t: X_t^\reg \hookrightarrow X_t$ is the inclusion map for each $t$.
	
	Now, assume that $K_{X_t}$ is trivial for all $t$ close to $0$. 
	Since $(X_t)_{t \in \BD}$ are Calabi--Yau varieties, the map $t \mapsto h^0(X_t, K_{X_t})$ is constantly equal to $1$.
	As a direct consequence of Grauert's theorem \cite{Grauert_1960}, the direct image sheaf $\pi_\ast (K_{\CX / \BD})$ is locally free and 
	\begin{equation}\label{eq:Grauert}
		\pi_\ast(K_{\CX / \BD}) \otimes k(0) \simeq H^0(X_0, K_{X_0}),
	\end{equation}
	where $k(0)$ is the residue field at $0$. 
	Let $\Om_0$ be a nowhere vanishing trivialization section of $K_{X_0}$ on $X_0$.
	By the isomorphism (\ref{eq:Grauert}), $\Om_0$ descends to an element $s_0$ of $\pi_\ast(K_{\CX / \BD}) \otimes k(0) = \res{\pi_\ast(K_{\CX / \BD})}{0}$.
	Since every line bundle over $\BD$ is trivial, we can extend the vector $s_0 \in \res{\pi_\ast (K_{\CX / \BD})}{0}$ to a non-vanishing section $s$ of $\pi_\ast(K_{\CX / \BD})$ after shrinking $\BD$. 
	Then $s$ is identified with a section $\Om$ of $K_{\CX / \BD}$ which has the relation $\res{\Om}{X_0} = \Om_0$.
	Since $\Om_0$ is nowhere vanishing, the section $\Om$ is nowhere zero on a neighborhood of the central fibre. 
\end{proof}

\begin{rmk}
	One can prove that under the assumption $X_0$ has canonical singularities, then $\CX$ is normal, $\BQ$-Gorenstein in a neighborhood of $X_0$ and moreover $X_t$ has canonical singularities for $t$ close to $0$.
	However, we will not use that result.
\end{rmk}

\subsection{Families of Calabi--Yau varieties}
It is thus legitimate to work in the following setting:
\begin{taggedbigset}{(CY)}\label{set:family_CY_var}
	Let $\pi: (\CX, \om) \ra \BD$ be a family of compact hermitian varieties satisfying the geometric setting \ref{geoset:irr_family}. 
	Suppose that 
	\begin{enumerate}
		\item $\CX$ is normal and $K_{\CX}$ is trivial;
		\item for all $t \in \BD$, $X_t$ has only canonical singularities.
	\end{enumerate}
\end{taggedbigset}

From Lemma~\ref{lem:AG_and_fibration} and the inversion of adjunction (cf. \cite[Theorem 5.50]{Kollar_Mori_1998}), Setting \ref{set:family_CY_var} implies following properties:
\begin{enumerate}[label=(\arabic*)]
	\item $\CX$ has canonical singularities;
	\item For all $t \in \BD$, $K_{X_t}$ is trivial.
\end{enumerate}

Let $\Om$ be a trivializing section of $K_{\CX / \BD}$.
Define the function $\gm_t$ on $X_t$ by the equation 
\[
\Om_t \w \overline{\Om_t} = \e^{-\gm_t} \om_t^n
\]
and $\gm_t$ also induces a function $\gm$ on $\CX$ near $X_0$.
From {\cite[Lemma 4.4]{DGG2020}}, we have the following uniform integrability property:
\begin{prop}\label{prop:Lp_estimate_of_CY_densities}
	Up to shrinking $\BD$, there exists $p > 1$ and $C > 0$ such that for all $t \in \BD$, we have
	\[
	\int_{X_t} \e^{-p \gm_t} \om_t^n < C. 
	\]
\end{prop}

On the other hand, if $K \Subset \CX^\reg$, then $\gm \in \CC^0(K)$ and it implies that 
\[
\int_{X_t} \e^{-\gm_t/n} \om_t^n 
\geq \e^{-\frac{\sup_K \gm_t}{n}} \int_{K_t} \om_t^n > c
\]
for some constant $c>0$ independent of $t$ close to $0$.
Therefore the canonical densities satisfy the integral bound (\ref{integral_bound:IB}) in Theorem~\ref{bigthm:uniform_a_priori_estimate}.

\smallskip

We are now ready to establish a uniform control of Chern--Ricci flat potentials in a family of Calabi--Yau varieties satisfying Setting \ref{set:family_CY_var} and the geometric assumption \ref{geoasm:L1_assumption}.

\begin{proof}[Proof of Theorem~\ref{bigthm:CY_family}]
	For all $t \in \BD$, $K_{X_t}$ is trivial and $K_{X_t} = \res{K_{\CX / \BD}}{X_t}$ as well.
	Therefore, one can find a non-vanishing section $\Om_t$ of $K_{X_t}$ satisfying $\Om_t = \res{\Om}{X_t}$.
	According to a recent work of Guedj and Lu \cite[Theorem E]{Guedj_Lu_3_2021}, for each $t \in \BD$, there exists a solution $(\vph_t, c_t) \in \lt( \PSH(X_t, \om_t) \cap L^\infty(X_t) \rt) \times \BR_{>0}$ which solves the complex Monge--Amp\`ere equation
	\[
	(\om_t + \ddc_t \vph_t)^n = c_t \Om_t \w \overline{\Om_t},
	\quad \text{and} \quad \sup_{X_t} \vph_t = 0.
	\]
	According to Theorem~\ref{bigthm:uniform_a_priori_estimate}, Proposition~\ref{bigprop:uniform_L1_isolated_singularities}, and Proposition~\ref{prop:Lp_estimate_of_CY_densities}, there is a uniform constant $C_{\MA}$ such that for all $t \in \BD_{1/2}$, one has
	\[
	c_t + c_t^{-1} + \norm{\vph_t}_{L^\infty} \leq C_{\MA}
	\]
	and this complete the proof of Theorem~\ref{bigthm:CY_family}.
\end{proof}

\bibliographystyle{smfalpha}
\bibliography{bib_Families_DCMAE}
\end{document}